\documentclass[11pt, letterpaper]{amsart}
\usepackage[colorlinks]{hyperref}
\AtBeginDocument{
  \hypersetup{
    linkcolor=blue,
    citecolor=green,
  }
}
\usepackage{mathscinet}
\usepackage{mathtools}
\usepackage{bm}
\usepackage{xcolor}
\usepackage[T1]{fontenc}
\makeatletter
\let\c@author\relax
\makeatother
\usepackage[backend=bibtex,style=numeric,sorting=nyt,
giveninits=true,isbn=false,url=false]{biblatex}
\renewbibmacro{in:}{\ifentrytype{article}{}{\printtext{\bibstring{in}\intitlepunct}}}
\DeclareFieldFormat*{title}{#1}
\usepackage[mathlines,pagewise]{lineno}
\newcommand*\patchAmsMathEnvironmentForLineno[1]{%
  \expandafter\let\csname old#1\expandafter\endcsname\csname #1\endcsname
  \expandafter\let\csname oldend#1\expandafter\endcsname\csname end#1\endcsname
  \renewenvironment{#1}%
     {\linenomath\csname old#1\endcsname}%
     {\csname oldend#1\endcsname\endlinenomath}}%
\newcommand*\patchBothAmsMathEnvironmentsForLineno[1]{%
  \patchAmsMathEnvironmentForLineno{#1}%
  \patchAmsMathEnvironmentForLineno{#1*}}%
\AtBeginDocument{%
\patchBothAmsMathEnvironmentsForLineno{equation}%
\patchBothAmsMathEnvironmentsForLineno{align}%
\patchBothAmsMathEnvironmentsForLineno{flalign}%
\patchBothAmsMathEnvironmentsForLineno{alignat}%
\patchBothAmsMathEnvironmentsForLineno{gather}%
\patchBothAmsMathEnvironmentsForLineno{multline}%
}

\usepackage{aliascnt}
\usepackage{cleveref}

\newtheorem{theorem}{Theorem}[section]
\Crefname{theorem}{Theorem}{Theorems}

\newaliascnt{lemma}{theorem}      
\newtheorem{lemma}[lemma]{Lemma}  
\aliascntresetthe{lemma}          
\Crefname{lemma}{Lemma}{Lemmas}

\newaliascnt{corollary}{theorem}      
\newtheorem{corollary}[corollary]{Corollary}  
\aliascntresetthe{corollary}          
\Crefname{corollary}{Corollary}{Corollaries}

\newaliascnt{proposition}{theorem}      
\aliascntresetthe{proposition}          
\Crefname{proposition}{Proposition}{Propositions}

\theoremstyle{definition}
\newaliascnt{definition}{theorem}
\newtheorem{definition}[definition]{Definition}
\aliascntresetthe{definition}
\Crefname{definition}{Definition}{Definitions}

\newaliascnt{notation}{theorem}

\aliascntresetthe{notation}
\Crefname{notation}{Notation}{Notations}

\theoremstyle{remark}
\newaliascnt{remark}{theorem}
\newtheorem{remark}[remark]{Remark}
\aliascntresetthe{remark}
\Crefname{remark}{Remark}{Remarks}

\crefname{equation}{}{}

\numberwithin{equation}{section}

\begin{document}

\title[Asymptotic behavior]{Asymptotic behavior at infinity of Weingarten surfaces}

\author{Aires E. M. Barbieri}
\address{Departamento de Geometr\'{i}a y Topolog\'{i}a, Instituto de Matem\'{a}ticas IMAG, Universidad de Granada}
\email{airesb@ugr.es}

\author{Jos\'{e} A. G\'{a}lvez}
\address{Departamento de Geometr\'{i}a y Topolog\'{i}a, Instituto de Matem\'{a}ticas IMAG, Universidad de Granada}
\email{jagalvez@ugr.es}

\author{Yuanyuan Lian}
\address{Departamento de Geometr\'{i}a y Topolog\'{i}a, Instituto de Matem\'{a}ticas IMAG, Universidad de Granada}
\email{lianyuanyuan.hthk@gmail.com; yuanyuanlian@correo.ugr.es}

\author{Kai Zhang}
\address{Departamento de Geometr\'{i}a y Topolog\'{i}a, Instituto de Matem\'{a}ticas IMAG, Universidad de Granada}
\email{zhangkaizfz@gmail.com; zhangkai@ugr.es}

\thanks{This research has been partially supported by PID2024-160586NB-I00, PID2023-150727NB-I00, the ``Maria de Maeztu'' Excellence Unit IMAG CEX2020-001105-M funded by MICIU/AEI/10.13039/501100011033, and the ``la Caixa'' Foundation (ID 100010434), fellowship LCF/BQ/DI24/12070007.}


\date{}



\begin{abstract}
We derive the asymptotic expansion at infinity for embedded ends of uniformly elliptic Weingarten surfaces with finite total curvature in $\mathbb{R}^3$, and we establish a maximum principle at infinity. Furthermore, we solve the Dirichlet problem for the uniformly elliptic Weingarten equation in dimension two on strictly convex bounded domains.
\end{abstract}

\maketitle

\section{Introduction}\label{S1}
Let $\Sigma\subset\mathbb{R}^3$ be an oriented immersed surface. $\Sigma$ is said to be a {\it uniformly elliptic Weingarten surface} if its principal curvatures $\kappa_1\geq \kappa_2$  satisfy
\begin{equation}\label{e1.6}
\kappa_2=f(\kappa_1),
\end{equation}
for a certain function $f\in C^{1,1}_{loc}[c,+\infty)$, where
 \begin{equation}\label{e1.6-2}
 -\frac{1}{\Lambda}\leq f'(t)\leq  -\Lambda,~\forall ~t\geq c,
 \end{equation}
for some constant $0<\Lambda<1$, and $f(c)=c$. Note that constant mean curvature (CMC) surfaces and minimal surfaces are included in this class. Moreover, $\Sigma$ is called of {\it minimal type} if
 \begin{equation}\label{e1.6-t}
f(0)=0.
 \end{equation}

The global study of uniformly elliptic Weingarten surfaces has been developed by many authors (see, for instance, \cite{MR2652214,MR86338,MR74857,EspinarMesa,MR4684292,MR4182894,MR4237967,MR63082, MR40042,MR1013786,MR4942311,MR1262209,MR1364263,MR1738404,MR108824}). However, some important problems still remain open.

Observe that the theory of uniformly elliptic Weingarten surfaces can be regarded as the natural fully nonlinear counterpart of the theory of CMC surfaces. From this perspective, one of the fundamental problems in the theory is to extend to this fully nonlinear setting some of the main global theorems of CMC surface theory. In this regard, it is important to note that, although certain features do extend without essential difficulty—such as invariance under ambient isometries, the validity of a maximum principle, and the existence of basic examples including spheres, cylinders, planes, and rotationally symmetric surfaces—many properties are lost in the passage to the Weingarten setting. For instance, the Weierstrass representation for minimal surfaces, the harmonicity of the Gauss map for CMC surfaces, and, more generally, the variational structure of the associated PDE no longer hold for Weingarten surfaces.

In this paper, we study the asymptotic behavior at infinity of uniformly elliptic Weingarten surfaces of minimal type. We obtain the expansion at infinity of each embedded end with finite total curvature, whose behavior depends strongly on the value of $f'(0)$. Observe that, in the case of minimal surfaces, this problem can be solved using the Weierstrass representation in terms of holomorphic data.
More precisely, it is well known that every embedded end of a minimal surface with finite total curvature is asymptotic to either a plane or a catenoid. Thus, for such a minimal surface, after an isometry, each end can be seen as a graph over a plane outside a ball $B_R$, centered at the origin and of radius $R$, of the form
$$
u(x)=d\log|x|+c+o(1),\qquad x\in B_R^c.
$$

This result is essential for the study of complete embedded minimal surfaces with finite total curvature and constitutes a fundamental tool for various theorems, such as, among many others, Schoen’s classical characterization of rotational minimal surfaces \cite{MR730928} and the classification of complete embedded minimal surfaces of genus zero with finite total curvature by L\'opez and Ros \cite{MR1085145}. It is also used, for instance, to study the moduli space of finite total curvature surfaces, to control the stability, and to analyze the rigidity of their periods.

Consequently, in order to extend some of the fundamental theorems of minimal surface theory to the fully nonlinear context, it is crucial to understand the behavior of finite total curvature ends in this setting. However, a Weierstrass representation is no longer available; therefore, for general uniformly elliptic Weingarten surfaces, the problem reduces to the study of the asymptotic behavior at infinity of solutions to their associated fully nonlinear (and non-uniformly) elliptic equations over exterior domains in $\mathbb{R}^2$.

Thanks to our description of the asymptotic behavior at infinity of the embedded ends of finite total curvature, we are also able to establish a maximum principle at infinity (cf. \cite{MR975123} for the case of minimal surfaces).
Our expansion at infinity and the associated maximum principle are used by Espinar and Mesa \cite{EspinarMesa} to give a Jorge-Meeks formula \cite{MR683761} for uniformly elliptic Weingarten surfaces of minimal type and also extend the classical Schoen characterization of rotational examples for minimal surfaces \cite{MR730928}. Additionally, our approach should help, for instance, to solve the exterior Dirichlet problem for \cref{e1.6}--\cref{e1.6-2}--\cref{e1.6-t}, to construct new complete examples and to study the moduli space of the complete embedded examples with finite total curvature.

We start our study in \Cref{S2} introducing some preliminaries on uniformly elliptic Weingarten surfaces of minimal type. In particular, we observe that, up to an isometry, every embedded end with finite total curvature can be seen as a graph $u$ over the exterior of a compact set of $\mathbb{R}^2$ with limit unit normal $N_{\infty}=(0,0,1)$ at infinity. 

We devote \Cref{S3} to solving the Dirichlet problem in strictly convex bounded domains, which constitutes a key tool for the subsequent results and is of independent interest. Classical existence results for fully nonlinear elliptic equations in two dimensions (e.g., \cite[Theorem 17.12]{MR1814364}) do not directly apply to the uniformly elliptic Weingarten equations of minimal type. If we henceforth assume that all graphs are upward-oriented, we show:
\begin{theorem}\label{th5.1}
Let $\Omega \subset \mathbb{R}^2$ be a strictly convex, bounded domain of class $C^{2,\alpha}$, and let $\varphi \in C^{2,\alpha}(\partial \Omega)$, with $0 < \alpha < 1$. Then there exists a unique solution $u \in C^{2,\beta}(\bar{\Omega})$ (for some $0 < \beta \leq \alpha$) to \cref{e1.6}--\cref{e1.6-2}--\cref{e1.6-t} where $f\in C_{loc}^{0,1}[0,+\infty)$, subject to the boundary condition
\begin{equation*}\label{e5.1}
u = \varphi \qquad \text{on } \partial \Omega.
\end{equation*}
Moreover,
\begin{equation}\label{e5.0}
\|u\|_{C^{2,\beta}(\bar{\Omega})} \leq C,
\end{equation}
where $C$ depends only on $\Lambda$, $\alpha$, $\max \kappa_{\partial \Omega}$, $\min \kappa_{\partial \Omega}$, $\|\partial \Omega\|_{C^{2,\alpha}}$, and $\|\varphi\|_{C^{2,\alpha}(\partial \Omega)}$. Here, $\kappa_{\partial \Omega}$ denotes the curvature of $\partial \Omega$.
\end{theorem}
\medskip

An important ingredient in order to obtain the expansion at infinity will be to prove the constant sign property of every graph satisfying \cref{e1.6}--\cref{e1.6-2}--\cref{e1.6-t} in the exterior of a ball $B_R\subset\mathbb{R}^2$ of radius $R>0$. More concretely, we show in \Cref{S4}:
\begin{theorem}\label{th1.1}
Let $u$ define a uniformly elliptic Weingarten graph of minimal type over $B_R^c$, with $f\in C_{loc}^{0,1}[0,+\infty)$, and with limit unit normal $N_{\infty}=(0,0,1)$. Then, up to a vertical translation, either
\begin{equation}\label{e1.12}
u(x)>0\quad\mbox{ or }\quad u(x)<0
\end{equation}
for all $x\in B_R^c$.

Moreover, $u_{\infty}:=\lim_{x\to \infty} u(x)$ exists in $\mathbb{R}\cup\{\pm\infty\}$.
\end{theorem}
\medskip

In \Cref{S5,S6} we prove our main results. First, we show the asymptotic behavior at infinity of every embedded end of finite total curvature:

\begin{theorem}\label{th1.3}
Let $u$ define a uniformly elliptic Weingarten graph of minimal type over $B_R^c$ with $u>0$ and limit unit normal $N_{\infty}=(0,0,1)$. Then
\begin{equation*}\label{e1.1}
\left\{
  \begin{aligned}
&u(x)\simeq |x|^{1+f'(0)}~~
&\mbox{or}\quad u_{\infty}\in \mathbb{R},~~\quad&\mbox{if}~~-1<f'(0)<0;\\
&u(x)\simeq \log |x|~~
&\mbox{or}\quad u_{\infty}\in \mathbb{R},~~\quad&\mbox{if}~~f'(0)=-1;\\
&u_{\infty}\in \mathbb{R},\quad~&~&\mbox{if}~~f'(0)<-1,\\
  \end{aligned}
  \right.
\end{equation*}
where $u\simeq v$ means that there exists a positive constant $c$ such that
\begin{equation*}
\lim_{x\to \infty}\frac{u(x)}{v(x)}=c.
\end{equation*}

Moreover, if $f'(0)=-1$, there exist $d\geq 0$ and $c\in \mathbb{R}$ such that for any $0<\alpha<1$,
\begin{equation}\label{e1.10}
u(x)=d\log |x|+c+O(|x|^{-\alpha}),~\forall ~x\in B_R^c
\end{equation}
and
\begin{equation*}\label{e1.11}
|Du(x)|= O(|x|^{-1}), \quad
|D^2u(x)|= O(|x|^{-2}),~\forall ~x\in B_R^c.
\end{equation*}
Here, $f=O(g)$ means that there exists a constant $C$ such that $|f|\leq Cg$.
\end{theorem}
\medskip

Observe that a specially important case happens when the Weingarten relation \cref{e1.6} is symmetric, that it, \cref{e1.6} can be rewritten as a smooth relation between the mean curvature $\mathcal{H}$ and the Gaussian curvature $\mathcal{K}$. In such a case, we always have $f'(0)=-1$ (see \cite{MR4417394,MR1013786}, for instance). For these surfaces, the corresponding fully nonlinear operator is smoother what allows us to obtain a better control of the expansion at infinity.

This enables us to control the asymptotic distance between two disjoint embedded ends of finite total curvature:

\begin{theorem}\label{th1.2}
Let $u,\tilde{u}$ define two graphs satisfying \cref{e1.6}--\cref{e1.6-2}--\cref{e1.6-t} over the exterior of some ball $B_R$, with $u,\tilde{u}>0$ and limit unit normal $N_{\infty}=(0,0,1)$.
Assume $f'(0)=-1$ and
\begin{equation*}
u\geq \tilde u\quad\mbox{ in }B_R^c.
\end{equation*}
Then there exists a constant $c_0>0$ such that
\begin{equation}\label{e1.9}
\frac{u(x)-\tilde{u}(x)}{\log |x|}\to c_0\quad\mbox{as}~~ x\to \infty
\end{equation}
or there exists a constant $c_0\geq 0$ such that
\begin{equation}\label{e1.5}
u(x)-\tilde u(x)\to c_0\quad\mbox{as}~~ x\to \infty.
\end{equation}

Moreover, if \cref{e1.5} holds, there exists $r_0>R$ such that
\begin{equation}\label{e1.6-3}
\min_{\partial B_r} (u-\tilde u)\leq c_0\leq \max_{\partial B_r} (u-\tilde u),~\forall ~ r>r_0.
\end{equation}
\end{theorem}
\medskip

Thus, as a consequence, we obtain the following strong comparison principle at infinity:

\begin{corollary}\label{co1.1}
Let $u,\tilde{u}$ define two graphs satisfying  \cref{e1.6}--\cref{e1.6-2}--\cref{e1.6-t} over the exterior of some ball $B_R$, with $u,\tilde{u}>0$ and limit unit normal $N_{\infty}=(0,0,1)$.
Assume $f'(0)=-1$,
\begin{equation*}
u\geq \tilde u\quad\mbox{ in }B_R^c
\end{equation*}
and there exists a sequence of points $\left\{x_n\right\}\subset B_R^c$ with $x_n\to \infty$ such that
\begin{equation*}
u(x_n)-\tilde u(x_n)\to 0\quad\mbox{as}~~n\to \infty.
\end{equation*}
Then, both graphs agree, that is,
\begin{equation*}
u\equiv\tilde u \quad\mbox{in}~~B_R^c.
\end{equation*}
\end{corollary}
\medskip

Finally, in the appendix, we show how our methods can be used to extend some of the main results in \cite{MR4684292} concerning elliptic Weingarten surfaces with a differentiable relation between their mean curvature $\mathcal{H}$ and Gaussian curvature $\mathcal{K}$. In particular, we generalize Bernstein-type results and the existence of curvature estimates to more general classes of elliptic Weingarten surfaces.
\section{Preliminaries}\label{S2}
Let $\Sigma$ be an oriented immersed surface in $\mathbb{R}^3$ with Gauss map $N:\Sigma\to\mathbb{S}^2$, satisfying \cref{e1.6}--\cref{e1.6-2}--\cref{e1.6-t}. Then, its Gaussian curvature $\mathcal{K}=\kappa_1\kappa_2\leq0$ and
\begin{equation}\label{ec1}
2|\mathcal{K}|\leq \kappa_1^2+\kappa_2^2\leq\frac{2}{\Lambda}|\mathcal{K}|.
\end{equation}
This indicates that $N$ is a quasiregular mapping (see \cite{MR1814364,MR452746}).

\begin{lemma}\label{lema1}
Let $\Sigma \subset \mathbb{R}^3$ be a complete, embedded, uniformly elliptic Weingarten surface of minimal type, possibly with compact boundary. Assume that $\Sigma$ has finite total curvature, that is,
$$
\int_{\Sigma}|\mathcal{K}|\,dA<\infty.
$$
Then $\Sigma$ has finitely many ends, each of which is properly embedded and possesses a limit unit normal at infinity. Moreover, after possibly truncating an end, it can be represented, up to an isometry, as the graph of a function defined over the complement of a compact set $K \subset \mathbb{R}^2$.
\end{lemma}
\proof
Since $\Sigma$ has finite total curvature, we have from \cite{MR94452} that $\Sigma$ has a finite number of ends and each of them is conformally equivalent to the puncture unit disk $\bar{\mathbb{D}}^{\ast}=\{z\in\mathbb{C}:\ 0<|z|\leq 1\}$. Thus, fixed an end $E$ of $\Sigma$, using \cref{ec1}, we get from \cite{MR906393} that $E$ is properly embedded and has a limit unit normal at infinity, that is, there exists $N_{\infty}:=\lim_{z\to 0}N(z)$ when $E$ is conformally parameterized as $\bar{\mathbb{D}}^{\ast}$.

Up to an isometry of $\mathbb{R}^3$ we can assume that $N_{\infty}=(0,0,1)$. Thus, shrinking $E$ if necessary, we can assume that the third coordinate of the Gauss map $N_3\geq 1/2$ in $E$. Hence, the vertical projection on the plane $x_3=0$ is a local diffeomorphism, that is, $E$ can be locally written as a graph $(x_1,x_2,u(x_1,x_2))$. Since $E$ is properly embedded and the gradient $|Du|\leq 1/2$, a standard topological argument gives us that, shrinking $E$ again if necessary, the end is a graph over the exterior of a compact set of $\mathbb{R}^2$.
~\qed\bigskip

Hence, to analyze the behavior of embedded ends of finite total curvature, we will concentrate on the study of graphs outside a compact set in $\mathbb{R}^2$.

Our attention now turns to the fully nonlinear elliptic partial differential equations satisfied by the surfaces under consideration. We start with the definition of uniform ellipticity (see \cite[Definition 2.1 and Lemma 2.2]{MR1351007}):
\begin{definition}\label{de2.1}
Let $\mathcal{S}^2$ be the set of $2\times 2$ symmetric matrices. An operator $F:\mathbb{R}^2\times \mathcal{S}^2\to \mathbb{R}$ is called uniformly elliptic if there exist positive constants $\bar{\lambda}\leq \bar{\Lambda}$ such that for any $p\in \mathbb{R}^2$ and $M,N\in \mathcal{S}^{2}$,
\begin{equation*}\label{e1.7}
\begin{aligned}
F(p,M)-F(p,N)
\leq \bar \Lambda |(M-N)^+|-\bar\lambda |(M-N)^-|,
\end{aligned}
\end{equation*}
where $(M-N)^+$ and $(M-N)^-$ denote the positive part and negative part of $M-N$ respectively. The constants $\bar{\lambda},\bar{\Lambda}$ are called ellipticity constants.

We will frequently use the following structure condition, which comprises the Lipschitz continuity of $F$ with respect to $p$: there exist positive constants $\bar{\lambda}\leq \bar{\Lambda}$ and $\mu$ such that for any $p,q\in \mathbb{R}^2$ and $M,N\in \mathcal{S}^{2}$,
\begin{equation}\label{e1.7-2}
\begin{aligned}
F(p,M)-F(q,N)
\leq \bar \Lambda |(M-N)^+|-\bar\lambda |(M-N)^-|+\mu|p-q|.
\end{aligned}
\end{equation}
We also refer to $\mu$ as the ellipticity constant for simplicity.
\end{definition}
\begin{remark}\label{re1.2}
Note that \cref{e1.7-2} is equivalent to the following: $F\in C^{0,1}(\mathbb{R}^2\times \mathcal{S}^{2})$ (i.e., Lipschitz continuous) and there exist positive constants $\bar{\lambda}\leq \bar{\Lambda}$ and $\mu$ such that
\begin{equation}\label{e1.8}
\begin{aligned}
\bar{\lambda} I \leq F_{M}(p,M)\leq \bar{\Lambda} I, \quad |F_{p}(p,M)|\leq \mu,~~\mbox{for}~~a.e.~~ p\in \mathbb{R}^2,~M\in\mathcal{S}^2,
\end{aligned}
\end{equation}
where $I$ is the unit matrix. As usual, $F_M$ and $F_p$ denote the first derivatives of $F$ with respect to the matrix $M$ and the vector $p$ respectively, i.e., $F_M$ is a matrix and $F_p$ is a vector whose elements are given by
\begin{equation*}
(F_M)_{ij}:=F_{M_{ij}}:=\frac{\partial F}{\partial M_{ij}}, \quad 1\leq i,j\leq 2; \quad
(F_p)_i:=F_{p_i}:=\frac{\partial F}{\partial p_{i}}, \quad 1\leq i\leq 2.
\end{equation*}

Indeed, for any $p,q\in \mathbb{R}^2$ and $M,N\in \mathcal{S}^{2}$, by the Newton-Leibniz formula,
\begin{equation}\label{e2.11}
\begin{aligned}
F(p,M)-F(q,N)=\int_{0}^{1}\left( F_{p_i}(\xi)(p-q)_i
+F_{M_{ij}}(\xi)(M-N)_{ij}\right)dt,
\end{aligned}
\end{equation}
where $\xi=\left(q+t(p-q),N+t(M-N)\right).$ Thus, it follows from \cref{e2.11} that \cref{e1.7-2} and \cref{e1.8} are equivalent.

\end{remark}

\begin{remark}\label{re1.4}
Note that the uniform ellipticity of a PDE is different from that of a Weingarten surface (i.e. \cref{e1.6-2}). Indeed, \cref{e1.8} means that the two eigenvalues $\lambda_1,\lambda_2$ of $F_M$ lie between two positive constants. Instead, \cref{e1.6-2} means that the product of $\lambda_1$ and $\lambda_2$ lies between two positive constants if we assume the boundedness of $|p|$ additionally (see \cite[Section 2.4 in Chapter V, pp. 128-129]{MR1013786}). Hence, \cref{e1.8} (or equivalently \cref{e1.7-2}) is stronger than \cref{e1.6-2}.

On the other hand, if $\lambda_1,\lambda_2$ are bounded, \cref{e1.6-2} immediately implies \cref{e1.8}. That is, we obtain a uniformly elliptic PDE if we have curvature estimates, which is true in fact (see \Cref{le2.1} and \cite[pp. 1913-1915]{MR4684292}).
\end{remark}
\medskip

We will also need the following interior $C^{2,\alpha}$ regularity result (see \cite[Corollary 1.2]{MR5005582}).
\begin{lemma}\label{le2.2}
Let $R>0$ and suppose that $u\in C^2(B_R)$ is a solution of
\begin{equation*}
F(Du,D^2u)=0\quad\mbox{in}~~B_{R}\subset \mathbb{R}^2,
\end{equation*}
where $F$ is uniformly elliptic. Then $u\in C^{2,\alpha}(\bar{B}_{R/2})$ for some $0<\alpha<1$ (depending only on $\bar{\lambda},\bar{\Lambda}$) and
\begin{equation*}
\begin{aligned}
&\sum_{k=1}^{2}R^k\|D^ku\|_{L^{\infty}(\bar{B}_{R/2})}+R^{2+\alpha}[D^2u]_{C^{\alpha}(\bar{B}_{R/2})}\leq C\left(\|u\|_{L^{\infty }(B_R)}+R^2|F(0,0)|\right), \quad\\
\end{aligned}
\end{equation*}
where $C$ depends only on $\bar{\lambda},\bar{\Lambda}$ and $R\mu$.
\end{lemma}
\medskip	

Another important property of surfaces satisfying \cref{e1.6}--\cref{e1.6-2}--\cref{e1.6-t} is the existence of curvature estimates, which were proven in \cite{MR4684292} (see also \Cref{th8.1}):
\begin{lemma}\label{le2.1}
Let $\Sigma$ be a complete uniformly elliptic Weingarten surface of minimal type, with boundary $\partial \Sigma$. Here, we only assume $f\in C_{loc}^{0,1}[0,+\infty)$. Suppose that its Gauss map lies in an open hemisphere of $\mathbb{S}^2$. Then
\begin{equation}\label{e1.3}
|\sigma (p)| \leq \frac{C}{d(p,\partial \Sigma)},~\forall ~p\in \Sigma,
\end{equation}
where $C$ depends only on $\Lambda$. Here, $d(p,\partial \Sigma)$ denotes the intrinsic distance from $p\in\Sigma$ to $\partial \Sigma$ and $|\sigma(p)|$ denotes the norm of its second fundamental form at $p$.
\end{lemma}

Next, we show that if $u$ defines a graph satisfying \cref{e1.6}--\cref{e1.6-2}--\cref{e1.6-t} then $u$ is indeed a solution of a fully nonlinear uniformly elliptic equation when $\|u\|_{C^2}$ is bounded. This has been proved in \cite[pp. 1913-1915]{MR4684292}. For the readers' convenience, we give the proof below. First, we introduce the following elementary lemma.
\begin{lemma}\label{le2.7}
Suppose that $G\in C^2(\mathbb{R}^n)$ and $G\geq 0$. Then for any $r\geq 1$,
\begin{equation}\label{e2.0}
|DG|\leq C\sqrt{G}\quad\mbox{in}~~B_r,
\end{equation}
where $C$ depends only on $\|G\|_{L^{\infty}(B_{2r})}$ and $\|D^2G\|_{L^{\infty}(B_{2r})}$.
\end{lemma}
\proof Fix $r\geq 1$ and $x\in B_r$. Without loss of generality, we assume $|DG(x)|>0$. Denote $K=\|D^2G\|_{L^{\infty}(B_{2r})}$. By the Taylor formula and noting $G\geq 0$, for any $h\in \bar B_{r}$, there exists $\xi\in B_r$ such that
\begin{equation*}
  \begin{aligned}
0\leq  G(x+h)=G(x)+DG(x)\cdot h+\frac{1}{2} h^TD^2G(\xi)h\leq G(x)+DG(x)\cdot h+\frac{1}{2} K|h|^2.\\
  \end{aligned}
\end{equation*}

If $|DG(x)|< Kr$, take $h=-DG(x)/K$. Then
\begin{equation*}
0\leq G(x)-\frac{|DG(x)|^2}{K}+\frac{1}{2} \frac{|DG(x)|^2}{K}.
\end{equation*}
Thus,
\begin{equation}\label{e2.40}
|DG(x)|\leq \sqrt{2K}\sqrt{G(x)}.
\end{equation}

If $|DG(x)|\geq Kr$, take $h=-rDG(x)/|DG(x)|$. Then
\begin{equation*}
0\leq G(x)-r|DG(x)|+\frac{1}{2} Kr^2.
\end{equation*}
Thus,
\begin{equation*}
|DG(x)|\leq \frac{G(x)}{r}+ \frac{1}{2} Kr\leq G(x)+\frac{1}{2}|DG(x)|.
\end{equation*}
That is,
\begin{equation}\label{e2.41}
|DG(x)|\leq 2G(x).
\end{equation}

By combining \cref{e2.40} and \cref{e2.41}, we arrive at the conclusion.~\qed\bigskip

Next, we prove the uniform ellipticity of \cref{e1.6}.
\begin{lemma}\label{le2.6}
Let $u$ define a solution of \cref{e1.6}--\cref{e1.6-2} in a domain $\Omega\subset \mathbb{R}^2$; where we only assume that $f\in C^{0,1}_{loc}([c,+\infty))$ and $f(c)=c\geq 0$. Then we can rewrite \cref{e1.6} in such a way that $u$ is seen as a solution to a fully nonlinear elliptic equation:
\begin{equation}\label{e2.3}
F(Du,D^2u)= 0\quad\mbox{in}~~\Omega.
\end{equation}
Note that $f$ may be not defined at $0$. To ensure that $F$ is well defined for all $(p,M)\in \mathbb{R}^2\times \mathcal{S}^2$, we extend the domain of $f$ from $[c,+\infty)$ to $(-\infty,+\infty)$ by reflecting the graph of $f$ with respect to the line $\kappa_2=\kappa_1$ (i.e., $f\circ f=\mathrm{Id}$).

Moreover, if $\|u\|_{C^2(\bar{\Omega})}\leq K$ for some positive constant $K$, $F$ can be enhanced to be uniformly elliptic (in the sense of \Cref{e1.7-2}).
\end{lemma}
\proof This lemma can be deduced from  \cite[pp. 1913-1915]{MR4684292}. For the readers' convenience, we give the proof below. Let us consider the mean curvature and the Gaussian curvature of the graph:
\begin{equation*}
\mathcal{H}=\left(\kappa_1+\kappa_2\right)/2, \quad \mathcal{K}=\kappa_1\kappa_2.
\end{equation*}
Then \cref{e1.6} can be rewritten as
\begin{equation*}
\mathcal{H}-\sqrt{\mathcal{H}^2-\mathcal{K}}
-f\left(\mathcal{H}+\sqrt{\mathcal{H}^2-\mathcal{K}}\right)=0.
\end{equation*}
Hence, $u$ is a solution of the following fully nonlinear equation:
\begin{equation}\label{e2.24}
F(Du,D^2u)= 0\quad\mbox{in}~~\Omega,
\end{equation}
where
\begin{equation}\label{e4.2}
F(p,M):=\mathcal{H}-\sqrt{\mathcal{H}^2-\mathcal{K}}
-f\left(\mathcal{H}+\sqrt{\mathcal{H}^2-\mathcal{K}}\right),
\end{equation}
and $\mathcal{H}$, $\mathcal{K}$ are functions of $(p,M)$ given by
\begin{equation*}
\begin{aligned}
\mathcal{H}(p,M):=\frac{(1+p_2^2)M_{11} - 2 p_1p_2 M_{12}  + (1+p_1^2)M_{22}}{2(1+|p|^2)^{3/2}}, \quad
\mathcal{K}(p,M):= \frac{\det M}{(1+|p|^2)^2}.
\end{aligned}
\end{equation*}
Obviously, $F$ is a continuous function. The ellipticity of $F$ is well-known (see \cite[p. 129]{MR1013786}).

Next, we prove the uniform ellipticity of $F$. Define for $r\geq 1$
\begin{equation}\label{e2.13}
\mathcal{B}_r:=\left\{(p,M)\in \mathbb{R}^2\times \mathcal{S}^2: |p|+|M|< r\right\}
\end{equation}
and
\begin{equation}\label{e2.2}
A_r:=\left\{(p,M)\in \mathcal{B}_{r}: \mathcal{H}^2-\mathcal{K}=0\right\}=\left\{(p,M)\in \mathcal{B}_{r}:M=\nu
\begin{pmatrix}
1+p_1^2&p_1 p_2\\
p_1 p_2&1+p_2^2
\end{pmatrix},\nu\in\mathbb{R}
\right\}.
\end{equation}
Note that $A_r$ is a submanifold of lower dimension. By the definition of $F$, we have $F\in C^{0,1}_{loc}$ in $\mathcal{B}_{r}\backslash A_r$.

We have known that $F$ is elliptic and to prove the uniform ellipticity, we only need to show (cf. \Cref{re1.2} and \Cref{re1.4})
\begin{equation}\label{e2.6}
|F_\omega|\leq C\quad a.e.~\mbox{in}~~\mathcal{B}_{r}\backslash A_r,~\forall ~\omega\in \left\{p_1,p_2,M_{11},M_{12},M_{22}\right\},
\end{equation}
where $F_\omega:=\partial F/\partial \omega$ and $C$ depends on $r$. By a direct calculation,
\begin{equation}\label{e2.12}
F_\omega=(1-f')\mathcal{H}_\omega
-\frac{1}{2}(1+f')\frac{\left(\mathcal{H}^2-\mathcal{K}\right)_\omega}
{\sqrt{\mathcal{H}^2-\mathcal{K}}}\quad a.e.~\mbox{in}~~\mathcal{B}_{r}\backslash A_r.
\end{equation}
Let us see that there exists a constant $C$ (depending only on $r$) such that
\begin{equation}\label{e2.26}
\left|\frac{\left(\mathcal{H}^2-\mathcal{K}\right)_\omega}{\sqrt{\mathcal{H}^2-\mathcal{K}}}\right|\leq C \quad\mbox{in}~~\mathcal{B}_{r}\backslash A_r, ~\forall ~\omega\in \left\{p_1,p_2,M_{11},M_{12},M_{22}\right\}.
\end{equation}

We have
\begin{equation*}
\mathcal{H}^2-\mathcal{K}=\frac{G(p,M)}{4
   \left(1+|p|^2\right)^3},
\end{equation*}
where
\begin{equation*}
G(p,M)=\left(\left(1+p_2^2\right) M_{11}-2 p_1 p_2 M_{12}+\left(1+p_1^2\right) M_{22}\right)^2
-4 \left(1+p_1^2+p_2^2\right) \left(M_{22} M_{11}-M_{12}^2\right).
\end{equation*}

Note that $G\geq 0$ is a polynomial. By \Cref{le2.7},
\begin{equation*}
|DG|\leq C\sqrt{G}\quad\mbox{in}~~\mathcal{B}_{r},
\end{equation*}
where $C$ depends only on $r$. From this inequality, we infer that  \cref{e2.6} holds. That is, $F\in C^{0,1}(\mathcal{B}_{r})$ and $F$ is uniformly elliptic in $\mathcal{B}_{r}$ (in the sense of \Cref{de2.1}) with ellipticity constants $\bar{\lambda},\bar{\Lambda}$ and $\mu$ depending on $r$.

If $\|u\|_{C^2(\bar{\Omega})}\leq K$, we have
\begin{equation*}
\left\{(Du(x),D^2u(x)): x\in \bar \Omega\right\}\subset \mathcal{B}_{K},
\end{equation*}
and we can define a new operator $F$ as follows
\begin{equation*}
\bar F(p,M):=\left\{
  \begin{aligned}
    &F(p,M)\quad&&\mbox{if}~~(p,M)\in  \mathcal{B}_{2K+1};\\
    &M_{11}+M_{22}\quad&&\mbox{if}~~(p,M)\in  \mathcal{B}_{4K+2}^c
  \end{aligned}
  \right.
\end{equation*}
such that $\bar F$ is a uniformly elliptic operator in $\mathbb{R}^2\times \mathcal{S}^2$. Clearly, $u$ is also a solution of \cref{e2.24} with $F$ replaced by $\bar F$. ~\qed\bigskip

\begin{remark}\label{re2.3}
From \cref{e2.12}, we know that if $\omega\in \left\{p_1,p_2\right\}$,
\begin{equation}\label{e4.3-2}
F_{\omega}(p,M)\to 0 \quad\mbox{for }~~(p,M)\in \mathcal{B}_{r}\backslash A_r,~M\to 0.
\end{equation}
The reason is that $\left(\mathcal{H}^2-\mathcal{K}\right)_\omega$ is quadratic in $M$.
\end{remark}

\begin{remark}\label{re2.4}
It follows from the proof that \Cref{le2.6} still holds if we relax the uniformly elliptic condition \cref{e1.6-2} to the elliptic condition:
\begin{equation*}
-\infty<\inf_{t\in [c,a]}f'(t)\leq \sup_{t\in [c,a]}f'(t)<0,~\forall ~a<+\infty.
\end{equation*}
\end{remark}
\medskip

In the rest of this section, we focus on the study of graphs defined in the exterior of a disk. So, let $u$ define such a graph in $B_R^c$ satisfying \cref{e1.6}--\cref{e1.6-2}--\cref{e1.6-t}. Since its Gauss map is a quasiregular mapping then, from \cite[Theorem 6]{MR454854}, there exists $\lim_{x\to\infty}Du(x)\in\mathbb{R}^2$. Thus, since our study focuses on the behavior of the graph at infinity, up to an isometry of $\mathbb{R}^3$, we may assume that $u$ is a graph in the exterior of a disk with  $\lim_{x\to\infty}|Du(x)|=0$.

\begin{lemma}\label{le2.0}
Let $u$ define a graph satisfying \cref{e1.6}--\cref{e1.6-2}--\cref{e1.6-t} in $B_R^c$ with $f\in C_{loc}^{0,1}[0,+\infty)$ and
\begin{equation}\label{ec2.10}
\lim_{x\to\infty}|Du(x)|=0.
\end{equation}
Then
\begin{equation}\label{e2.7}
|D^2u(x)|\leq C|\sigma(x)|\leq C|x|^{-1},~\forall ~x\in B_{2R}^c.
\end{equation}
\end{lemma}
\proof Given $R_0>R$, we get from \cref{ec2.10} and \Cref{le2.1} that
\begin{equation*}
|D^2u(x)|\leq C|\sigma(x)|\leq C|x-R|^{-1},~\forall ~x\in B_{R_0}^c.
\end{equation*}
Thus, \cref{e2.7} is obtained from the compactness of the annulus $\bar{B}_{R_0}-B_R$.
~\qed\bigskip

\begin{remark}\label{re2.2}
Since $|Du|$ and $|D^2u|$ are bounded, by \Cref{le2.6}, $u$ is a solution of a uniformly elliptic equation. By the $C^{2,\alpha}$ regularity (see \Cref{le2.2}), $u\in C^{2,\alpha}$ for some $0<\alpha<1$.
\end{remark}
\medskip

Now, we obtain the following second derivatives estimates for the fully nonlinear operator $F$.
\begin{lemma}\label{le2.4}
Let $F$ be the fully nonlinear elliptic operator in \cref{e4.2} where $f\in C^{1,1}_{loc}[0,+\infty)$ and satisfies \cref{e1.6-2}--\cref{e1.6-t}. Then for any $r>0$, there exists $C>0$ (depending on $r$) such that
\begin{equation}\label{e2.30-1}
\begin{aligned}
|F_{\omega\tau}|\leq C\quad a.e.~\mbox{in}~~\tilde{\mathcal{B}}_{r}\backslash \tilde A_r,
~\forall ~\omega\in \left\{p_1,p_2\right\},
\tau\in \left\{M_{11},M_{12},M_{22}\right\},
\end{aligned}
\end{equation}
where $\tilde{\mathcal{B}}_r$ and $\tilde A_r$ are defined as
\begin{equation*}
\tilde{\mathcal{B}_r}:=\left\{(p,M)\in \mathbb{R}^2\times \mathcal{S}^2: |p|+|M|< r, \quad \mathcal{K}\leq 0\right\}
\end{equation*}
and
\begin{equation*}
\begin{aligned}
\tilde{A}_r&:=\left\{(p,M)\in \tilde{\mathcal{B}}_{r}: \mathcal{H}^2-\mathcal{K}=0\right\}.
\end{aligned}
\end{equation*}

If $f'(0)=-1$, we have
\begin{equation}\label{e2.30}
|F_{\omega\tau}|\leq C\quad a.e.~\mbox{in}~~\tilde{\mathcal{B}}_{r}\backslash \tilde{A}_r,~\forall ~\omega,\tau\in \left\{p_1,p_2,M_{11},M_{12},M_{22}\right\}.
\end{equation}
\end{lemma}
\proof We calculate the derivatives of $F$ in $\tilde{\mathcal{B}}_{r}\backslash \tilde A_r$. From \cref{e2.12}, given $\omega\in \left\{p_1,p_2\right\},
\tau\in \left\{M_{11},M_{12},M_{22}\right\}$,  we have
\begin{equation}\label{e2.29}
  \begin{aligned}
F_{\omega\tau}=&-f''\mathcal{H}_\omega\left(\mathcal{H}_\tau+\frac{1}{2}\frac{\left(\mathcal{H}^2-\mathcal{K}\right)_\tau}
{\sqrt{\mathcal{H}^2-\mathcal{K}}}\right)+(1-f')H_{\omega\tau}
-\frac{f''}{2}\frac{\left(\mathcal{H}^2-\mathcal{K}\right)_\omega}
{\sqrt{\mathcal{H}^2-\mathcal{K}}}\frac{\left(\mathcal{H}^2-\mathcal{K}\right)_\tau}
{\sqrt{\mathcal{H}^2-\mathcal{K}}}\\
    &
-\frac{1}{2}(1+f')\frac{\left(\mathcal{H}^2-\mathcal{K}\right)_{\omega\tau}}
{\sqrt{\mathcal{H}^2-\mathcal{K}}}
+\frac{1+f'}{4}\frac{\left(\mathcal{H}^2-\mathcal{K}\right)_\omega}
{\mathcal{H}^2-\mathcal{K}}\frac{\left(\mathcal{H}^2-\mathcal{K}\right)_\tau}
{\sqrt{\mathcal{H}^2-\mathcal{K}}}.
  \end{aligned}
\end{equation}

With the aid of \cref{e2.26}, to prove \cref{e2.30-1}, we only need to show
\begin{equation}\label{e4.4}
\left|\frac{\left(\mathcal{H}^2-\mathcal{K}\right)_{\omega}}
{\mathcal{H}^2-\mathcal{K}}\right|\leq C, \quad
\left|\frac{\left(\mathcal{H}^2-\mathcal{K}\right)_{\omega\tau}}
{\sqrt{\mathcal{H}^2-\mathcal{K}}}\right|\leq C.
\end{equation}
First, note that $\mathcal{H}^2-\mathcal{K}$ is homogenous and quadratic in $M$. Thus, $\left(\mathcal{H}^2-\mathcal{K}\right)_{\omega}$ has the same property since $\omega\in \{p_1,p_2\}$. Then with the aid of \Cref{le2.0}, we have
\begin{equation*}
\left|\left(\mathcal{H}^2-\mathcal{K}\right)_{\omega}\right|\leq C|\sigma|^2.
\end{equation*}
In addition, note that $\mathcal{K}\leq 0$ and hence
\begin{equation}\label{e4.5}
\mathcal{H}^2-\mathcal{K}=\frac{1}{4}(\kappa_1^2+\kappa_2^2-2\kappa_1\kappa_2)
\geq \frac{1}{4}(\kappa_1^2+\kappa_2^2)=\frac{1}{4}|\sigma|^2.
\end{equation}
So, the first inequality in \cref{e4.4} holds. Similarly, the second holds as well. Therefore, we obtain \cref{e2.30-1}.

Next, we consider the case $f'(0)=-1$. From $|\mathcal{H}|\leq |\sigma|/\sqrt{2}$ and \cref{e4.5}, we have
\begin{equation*}\label{e2.27}
\left|\frac{\mathcal{H}}{\sqrt{\mathcal{H}^2-\mathcal{K}}}\right|\leq C.
\end{equation*}
Then by noting $f'(0)=-1$, there exists $\xi>0$ such that
\begin{equation*}
|1+f'|\leq |f''(\xi)|\left|\mathcal{H}+\sqrt{\mathcal{H}^2-\mathcal{K}}\right|\leq C\left|\mathcal{H}+\sqrt{\mathcal{H}^2-\mathcal{K}}\right|.
\end{equation*}
Hence,
\begin{equation*}\label{e2.28}
\left|\frac{1+f'}{\sqrt{\mathcal{H}^2-\mathcal{K}}}\right|\leq C\left|\frac{\mathcal{H}}{\sqrt{\mathcal{H}^2-\mathcal{K}}}\right|+C\leq C.
\end{equation*}
Then we know from \cref{e2.29} that
\begin{equation*}
|F_{\omega\tau}|\leq C,~\forall ~\omega,\tau\in \left\{p_1,p_2,M_{11},M_{12},M_{22}\right\}.
\end{equation*}
~\qed\bigskip

\begin{lemma}\label{le2.5}
Let $F$ be the fully nonlinear elliptic operator in \cref{e4.2} where $f\in C^{1,1}_{loc}[0,+\infty)$ and satisfies \cref{e1.6-2}--\cref{e1.6-t}. Then for any $r>0$, there exists $C>0$ (depending on $r$) such that
\begin{equation}\label{e2.10}
\begin{aligned}
|F(p,M)-F(q,M)|\leq C|M||p-q|\quad \mbox{in}~~\tilde{\mathcal{B}}_{r},
\end{aligned}
\end{equation}
where $\tilde{\mathcal{B}}_r$ is defined as above.

Moreover, if $f'(0)=-1$, $F\in C^{1,1}(\overline{\tilde{\mathcal{B}}_r})$.
\end{lemma}

\proof We first consider the case that at least one of $(p,M), (q,M)$ does not belong to $A_r$ and the line segment $L$ joining $(p,M)$ to $(q,M)$ is entirely contained in $\tilde{\mathcal{B}}_r\backslash A_r$. Then there exists $\xi\in L$ such that
\begin{equation}\label{e2.32}
F(p,M)-F(q,M)=F_p(\xi,M)\cdot (p-q).
\end{equation}
By the definition of $A_r$ (see \cref{e2.2}), $(\xi, tM)\in \tilde{\mathcal{B}}_r\backslash A_r$ for any $t\in (0,1]$. By \cref{e2.30-1}, we have
\begin{equation*}
|F_p(\xi,M)|\leq |F_p(\xi,M)-F_p(\xi,tM)|+|F_p(\xi,tM)|
\leq C(1-t)|M|+|F_p(\xi,tM)|.
\end{equation*}
Let $t\to 0$ and note \cref{e4.3-2}. Then
\begin{equation}\label{e2.33}
|F_p(\xi,M)|\leq C|M|.
\end{equation}
From \cref{e2.32} and \cref{e2.33},
\begin{equation}\label{e2.34}
|F(p,M)-F(q,M)|\leq C|M||p-q|.
\end{equation}

If the line segment $L$ is not entirely contained in $\tilde{\mathcal{B}}_r\backslash A_r$, we may decompose $L$ into several subsegments each lying entirely in $\tilde{\mathcal{B}}_r\backslash A_r$. Then by applying \cref{e2.34} to each subsegment and summing the resulting estimates, we obtain \cref{e2.34} as well. Finally, if $(p,M), (q,M)\in A_r$, we can also obtain \cref{e2.34} by an approximation procedure.

Next, we consider the case $f'(0)=-1$. We first show that $F_{\omega}$ is continuous in $\tilde{\mathcal{B}_r}$ for any $\omega\in \left\{p_1,p_2,M_{11},M_{12},M_{22}\right\}$. Indeed, since $\mathcal{K}\leq 0$ in $\tilde{\mathcal{B}}_r$,  $\mathcal{H}^2-\mathcal{K}\to 0$ implies $\mathcal{H}\to 0$ and hence $f'(\mathcal{H}+\sqrt{\mathcal{H}^2-\mathcal{K}})\to f'(0)$. Then by \cref{e2.12}, \cref{e2.26} and $f'(0)=-1$, we know that $F_{\omega}\to 2\mathcal{H}_{\omega}$ as $\mathcal{H}^2-\mathcal{K}\to 0$. Thus, $F_{\omega}$ extends continuously to $\tilde {A}_r$. That is, $F_{\omega}$ is continuous in $\tilde{\mathcal{B}}_r$. Then by an argument similar to the above, we conclude that $F_{\omega}$ is Lipschitz continuous in $\tilde{\mathcal{B}}_r$. That is, $F\in C^{1,1}(\overline{\tilde{\mathcal{B}}_r})$. ~\qed\bigskip

Finally, we improve the decay of $|Du|$ and $|D^2u|$.
\begin{lemma}\label{le2.3}
Consider a graph defined by a function $u$ in $B_R^c$ satisfying \cref{e1.6}--\cref{e1.6-2}--\cref{e1.6-t} where $f\in C_{loc}^{0,1}[0,+\infty)$. Assume that $\lim_{x\to\infty}|Du|=0$, then there exists $0<\alpha<1$ such that
\begin{equation*}
|Du(x)|\leq C|x|^{-\alpha}, \quad |D^2u(x)|\leq C|x|^{-1-\alpha},~\forall ~x\in B_{2R}^c.
\end{equation*}
\end{lemma}
\proof By \Cref{le2.6} and \Cref{le2.0}, $u$ is a solution of the uniformly elliptic equation \cref{e2.3}. Note that $F(p,0)\equiv 0$ by \cref{e4.2} and $f(0)=0$. Then $u$ is a solution of the following linear uniformly elliptic equation:
\begin{equation}\label{e2.9}
a^{ij}u_{ij}=0\quad\mbox{in}~~ \bar{B}_{2R}^c,
\end{equation}
where
\begin{equation*}
a^{ij}(x):=\int_{0}^{1}F_{M_{ij}}(Du(x),tD^2u(x)) dt.
\end{equation*}
Since $|Du|$ is bounded, by the gradient estimate for linear uniformly elliptic equations in dimension $2$ (see \cite[Theorem 2.6]{MR4792754}), there exists $0<\alpha<1$ such that
\begin{equation}\label{e2.8}
|Du(x)|\leq C|x|^{-\alpha},~\forall ~x\in \bar{B}_{2R}^c.
\end{equation}

In the following, we improve the decay of $D^2u$ at infinity. By \cref{e2.8},
\begin{equation}\label{e2.15}
|u(x)|\leq C|x|^{1-\alpha},~\forall ~x\in \bar{B}_{2R}^c.
\end{equation}

For any $x_0\in \bar{B}_{4R}^c$, let $r=(|x_0|-2R)/2$. Consider the equation \cref{e2.3} in $B_{r}(x_0)$. By \cref{e2.8} and \Cref{le2.0},
\begin{equation}\label{e2.35}
|Du|\leq Cr^{-\alpha}, \quad |D^2u|\leq Cr^{-1}\quad\mbox{in}~~B_{r}(x_0).
\end{equation}
Without loss of generality, we assume $Cr^{-\alpha}\leq 1$. Let
\begin{equation*}
\tilde{\mathcal{B}}:=\left\{(p,M)\in \mathbb{R}^2\times \mathcal{S}^2: |p|\leq 1, \quad |M|\leq Cr^{-1},\quad \mathcal{K}\leq 0\right\}\subset
\tilde{\mathcal{B}}_1.
\end{equation*}
Then
\begin{equation*}
\left\{(Du(x),D^2u(x)): x\in B_{r}(x_0)\right\}\subset \tilde{\mathcal{B}}.
\end{equation*}
In addition, by \Cref{le2.5},
\begin{equation*}
|F(p,M)-F(q,M)|\leq C|M||p-q|\leq Cr^{-1}|p-q|\quad \mbox{in}~~\tilde{\mathcal{B}}.
\end{equation*}

Hence, up to a proper extension, $u$ is a solution of a fully nonlinear uniformly elliptic equation with ellipticity constant $\mu=Cr^{-1}$. By \cref{e2.15} and the interior $C^{2,\alpha}$ regularity (see \Cref{le2.2}) for $u$ in $B_{r}(x_0)$, we have
\begin{equation*}
\begin{aligned}
|D^2u(x_0)|\leq \frac{C}{r^2}\|u\|_{L^{\infty}(B_r(x_0))}\leq Cr^{-1-\alpha},
\end{aligned}
\end{equation*}
where $C$ is independent of $r$. Therefore,
\begin{equation*}
|D^2u(x)|\leq  C|x|^{-1-\alpha},~\forall ~x\in \bar{B}_R^c.
\end{equation*}
\qed\bigskip

\section{Solvability of the Weingarten equation in a bounded convex domain}\label{S3}
In this section we obtain the solvability of the Weingarten equation for general strictly convex, bounded domains.

Along this section we search for a function $u$ defined in a strictly convex domain $\Omega$ whose graph satisfies
\begin{equation}\label{e5.1bis}
  \left\{
  \begin{aligned}
    &\kappa_2=f(\kappa_1)&&\quad\mbox{in}~~\Omega,\\
    &u=\varphi&&\quad\mbox{on}~~\partial \Omega,
  \end{aligned}
  \right.
\end{equation}
for a certain function $f\in C^{0,1}_{loc}[0,+\infty)$ satisfying \cref{e1.6-2}--\cref{e1.6-t}.

The theorem will be proved by means of the method of continuity, which is divided into several steps.

First, we observe that if $u$ is a solution of \cref{e5.1bis} then its graph is saddle, that is, has non positive Gaussian curvature. Thus, from \cite{MR3069540,zbMATH02587131} (see also \cite[p. 311]{MR1814364}) one has:
\begin{lemma}\label{le5.1}
Let $u\in C^{2}(\bar{\Omega})$ be a solution of \cref{e5.1bis}, where $\Omega$ is a  strictly convex, bounded, $C^2$ domain and $\varphi \in C^{2}(\partial \Omega)$. Then
\begin{equation}\label{e3.0}
\|u\|_{C^{0,1}(\bar{\Omega})}\leq C,
\end{equation}
where $C$ depends only on $\min \kappa_{\partial \Omega}$ and $\|\varphi\|_{C^{2}(\partial \Omega)}$.
\end{lemma}
A key ingredient is the following curvature estimate.
\begin{lemma}\label{le5.2}
Let $u\in C^{2}(\bar{\Omega})$ be a solution of \cref{e5.1bis}. Suppose that $\Omega$ is a strictly convex, bounded, $C^{2,\alpha}$  domain, and $\varphi \in C^{2,\alpha}(\partial \Omega)$ for some $0<\alpha<1$. Then the second fundamental form of the graph of $u$ is bounded, i.e.,
\begin{equation}\label{e3.2}
\|\sigma\|_{L^{\infty}(\bar{\Omega})}\leq C,
\end{equation}
where $C$ depends only on $\Lambda$, $\alpha$, $\max \kappa_{\partial \Omega}$, $\min \kappa_{\partial \Omega}$, $\|\partial \Omega\|_{C^{2,\alpha}}$ and $\|\varphi\|_{C^{2,\alpha}(\partial \Omega)}$.
\end{lemma}
\proof The proof is a standard combination of a blow-up process with the classification of entire solutions. We prove the lemma by contradiction.

Suppose that the conclusion is false. Then there exist positive constants $\Lambda$, $0<\alpha<1$, $K_1\leq K_2$ and $K$ and sequences of $f_n,u_n,\Omega_n,\varphi_n$ such that $u_n$ are solutions of
\begin{equation*}
  \left\{
  \begin{aligned}
    &\kappa_2=f_n(\kappa_1)&&\quad\mbox{in}~~\Omega_n;\\
    &u_n=\varphi_n&&\quad\mbox{on}~~\partial \Omega_n,
  \end{aligned}
  \right.
\end{equation*}
where $f_n$ are uniformly elliptic with $\Lambda$; the $\Omega_n$ are bounded convex domain with
\begin{equation*}
K_1\leq \kappa_{\Omega_n}\leq K_2,~\forall ~n
\end{equation*}
and
\begin{equation*}
\|\partial \Omega_n\|_{C^{2,\alpha}},~ \|\varphi_n\|_{C^{2,\alpha}(\partial \Omega_n)}\leq K.
\end{equation*}
Moreover, there exists a sequence of points $x_n\in \bar \Omega_n$ such that
\begin{equation*}\label{e5.3}
\lambda_n:=|\sigma_n(x_n)|=\max_{x\in \bar{\Omega}_n} |\sigma_n(x)| \geq n,
\end{equation*}
where $\sigma_n$ denotes the second fundamental form of the graph of $u_n$.

In the following, we divide the proof into two cases:~\\[2mm]
(i) Up to a subsequence, $d(x_n,\partial \Omega_n)\cdot \lambda_n\to \infty$.
\\[2mm]
In this case, we use the following blow-up:
\begin{equation*}
y:=\lambda_n(x-x_n), \quad v_n(y):=\lambda_n \left(u_n(x)-u_n(x_n)\right).
\end{equation*}
Then, after the rescaling, $v_n$ is a solution of
\begin{equation}\label{e5.4}
\kappa_2=g_n(\kappa_1)\quad\mbox{in}~~B_{R_n},
\end{equation}
where $R_n=d(x_n,\partial \Omega_n)\cdot \lambda_n\to \infty$ and
\begin{equation*}
g_n(t):=\lambda_n^{-1}f_n(\lambda_n t).
\end{equation*}
Hence, $v_n$ is a uniformly elliptic Weingarten graph of minimal type as well.

Furthermore, with the aid of the Lipschitz estimate for $u_n$, given by \Cref{le5.1}, we have
\begin{equation}\label{e5.5}
v_n(0)=0, \quad |Dv_n(y)|=|Du_n(x)|\leq C,~\forall ~y\in B_{R_n}.
\end{equation}
Moreover, the second fundamental form of $v_n$ is bounded. In fact,
\begin{equation*}
|\sigma_{v_n}(y)|=\lambda_n^{-1} |\sigma_n(x)|\leq 1,~\forall ~y\in B_{R_n}.
\end{equation*}
Hence,
\begin{equation}\label{e3.1}
\|D^2v_n\|_{L^{\infty}(B_{R_n})}\leq C.
\end{equation}
Therefore, \cref{e5.4} is a uniformly elliptic fully nonlinear equation for each $n$ (cf. \Cref{le2.6}).

By the interior $C^{2,\alpha}$ regularity (see \Cref{le2.2}), there exists $0<\beta<1$ such that
\begin{equation*}
\|v_n\|_{C^{2,\beta}(\bar{B}_R)}\leq C\|v_n\|_{L^{\infty}(B_{2R})},~\forall ~0<R<R_n/2,
\end{equation*}
where $C$ depends on $R$ but not on $n$. By \cref{e5.5},
\begin{equation*}
\|v_n\|_{L^{\infty}(B_{2R})}\leq CR.
\end{equation*}
Hence, $\|v_n\|_{C^{2,\beta}(\bar{B}_R)}$ are uniformly bounded for any fixed $R$. Since $R_n\to \infty$, up to a subsequence,  there exists $v\in C^{2,\beta}(\mathbb{R}^2)$ such that $v_n\to v$ in $C^2(\bar{B}_R)$ for any $0<R<\infty$.

Since $g_n(0)=0$ and $|g'_n|$ are uniformly bounded, up to a subsequence, there exists a Lipschitz function $g$ such that $g_n\to g$ uniformly and
\begin{equation*}
 -\frac{1}{\Lambda}\leq g'(t)\leq  -\Lambda,~\forall ~t\geq 0.
\end{equation*}
In addition, from \cref{e5.5} and \cref{e3.1}, we have
\begin{equation}\label{e3.5}
v(0)=0, \quad   \|Dv\|_{L^{\infty}(\mathbb{R}^2)}+\|D^2v\|_{L^{\infty}(\mathbb{R}^2)}\leq C.
\end{equation}
Hence, $v$ is a solution of the following uniformly elliptic equation
\begin{equation*}
  \kappa_2=g(\kappa_1)\quad\mbox{in}~~\mathbb{R}^2.
\end{equation*}

As in the proof of \Cref{le2.3} (see \cref{e2.9}), $v$ can be regarded as a solution of a linear uniformly elliptic equation:
\begin{equation}\label{e3.6}
a^{ij}u_{ij}=0\quad\mbox{in}~~ \mathbb{R}^2.
\end{equation}
From the interior $C^{1,\alpha}$ regularity (see \cite[Theorem 12.4]{MR1814364}), there exists $0<\alpha<1$ such that
\begin{equation*}
|Dv(x)-Dv(y)|\leq C\frac{\|v\|_{L^{\infty}(B_R)}}{R^{1+\alpha}}|x-y|^{\alpha},~\forall ~x,y\in B_{R/2},~\forall ~R>0.
\end{equation*}
By \cref{e3.5}, $\|v\|_{L^{\infty}(B_R)}\leq CR$. Hence, by letting $R\to \infty$ in the above equation, we conclude that $Dv$ is a constant vector. That is, $v$ is a linear function.

On the other hand,
\begin{equation*}
|\sigma_v(0)|=\lim_{n\to \infty} |\sigma_{v_n}(0)|=1,
\end{equation*}
where $\sigma_v$ denotes the second fundamental form of $v$. Therefore, we obtain a contradiction.

Next, we consider the case\\[2mm]
(ii) $d(x_n,\partial \Omega_n)\cdot \lambda_n\leq C_0$ for some positive constant $C_0$.  \\[2mm]
Let $\tilde{x}_n\in \partial \Omega_n$ such that $d(x_n,\partial \Omega_n)=|x_n-\tilde{x}_n|$. Consider a rotation in $\mathbb{R}^2$ with associated orthogonal matrix $A_n$ such that $A_n\nu_n=e_2=(0,1)$, where $\nu_n$ is the interior unit normal of $\partial \Omega_n$ at $\tilde{x}_n$. Then we  can consider the following blow-up:
\begin{equation*}
\begin{aligned}
&y=\lambda_n A_n(x-\tilde{x}_n), \quad \tilde{\Omega}_n=\lambda_n A_n(\Omega_n-\tilde{x}_n),\\
&v_n(y)=\lambda_n \left(u_n(x)-u_n(\tilde{x}_n)\right), \quad
\psi_n(y)=\lambda_n \left(\varphi_n(x)-u_n(\tilde{x}_n)\right).
\end{aligned}
\end{equation*}

As in the proof of the first case, $v_n$ is a solution of
\begin{equation}\label{e5.6}
\kappa_2=g_n(\kappa_1)\quad\mbox{in}~~\tilde{\Omega}_n, \quad v_n=\psi_n\quad\mbox{on}~~\partial \tilde{\Omega}_n,
\end{equation}
where
\begin{equation*}
g_n(t):=\lambda_n^{-1}f_n(\lambda_n t).
\end{equation*}
Similar to case (i), we have
\begin{equation*}\label{e5.7}
v_n(0)=0, \quad |Dv_n(y)|=|Du_n(x)|\leq C,~\forall ~y\in \tilde{\Omega}_n
\end{equation*}
and
\begin{equation*}
|\sigma_{v_n}(y)|\leq 1,~\forall ~y\in \tilde{\Omega}_n.
\end{equation*}
Hence,
\begin{equation*}
\|D^2v_n\|_{L^{\infty}(\tilde{\Omega}_n)}\leq C
\end{equation*}
and \cref{e5.6} are uniformly elliptic equations.

Additionally, note that
\begin{equation*}
\psi_n(0)=0, \quad\|D\psi_n\|_{L^{\infty}(\partial \tilde{\Omega}_n)}=
\|D\varphi_n\|_{L^{\infty}(\partial \Omega_n)}, \quad
 \|D^2\psi_n\|_{L^{\infty}(\partial \tilde{\Omega}_n)}=
 \lambda_n^{ -1}\|D^2\varphi_n\|_{L^{\infty}(\partial \Omega_n)}\to 0
\end{equation*}
and
\begin{equation*}
[D^2\psi_n]_{C^{\alpha}(\partial \tilde{\Omega}_n)}=
 \lambda_n^{-1-\alpha}[D^2\varphi_n]_{C^{\alpha}(\partial \Omega_n)}\to 0.
\end{equation*}
Then by the $C^{2,\alpha}$ regularity (see \cite[Corollary 8.4]{lian2020pointwise}), there exists $0<\beta\leq \alpha$ such that for any $R>0$, if $n$ is large enough, we have
\begin{equation*}
\|v_n\|_{C^{2,\beta}(\bar{B}_R\cap \tilde{\Omega}_n)}\leq C,
\end{equation*}
where $C$ depends on $R$ but not on $n$. Hence, $\|v_n\|_{C^{2,\beta}(\bar{B}_R\cap \tilde{\Omega}_n)}$ are uniformly bounded. 

In addition, for any $R>0$,
\begin{equation*}
\|\partial \tilde{\Omega}_n\cap B_R\|_{C^{2,\alpha}}\to 0\quad\mbox{as}~~n \to \infty.
\end{equation*}
Thus, up to a subsequence,  there exists $v\in C^{2,\alpha}(\overline{\mathbb{R}^2_+})$ such that $v_n\to v$ in $C^2(\bar{B}_R^+)$ for any $0<R<\infty$. Here, $\mathbb{R}^2_+=\{(x_1,x_2)\in\mathbb{R}^2:\ x_2>0\}$ and $B_R^+=B_R\cap\mathbb{R}^2_+$.

Similar to case (i), there exists a Lipschitz function $g$ with
\begin{equation*}
 -\frac{1}{\Lambda}\leq g'(t)\leq  -\Lambda,~\forall ~t\geq 0
\end{equation*}
such that $v$ is a solution of
\begin{equation*}
  \kappa_2=g(\kappa_1)\quad\mbox{in}~~\mathbb{R}^2_+.
\end{equation*}

By the Taylor formula for $\psi_n$, for any $R>0$ and $y\in \partial \tilde{\Omega}_n\cap B_R$, there exists $\xi\in \partial \tilde{\Omega}_n\cap B_R$ such that
\begin{equation*}
\psi_n(y)=D\psi_n(0)\cdot y+y^TD^2\psi_n(\xi)y.
\end{equation*}
Since $|D^2\psi_n|\to 0$, up to a subsequence, there exists $a\in \mathbb{R}$ such that $D_{y_2}\psi_n(0)\to a$. Then
\begin{equation*}
v(y)=ay_1\quad\mbox{on}~~\partial \mathbb{R}^2_+.
\end{equation*}

By applying the Liouville theorem in a half-space to \cref{e3.6} (see \cite[Theorem 2.3]{lian2025liouville}), $v$ must be linear. On the other hand, let
\begin{equation*}
  y_n=\lambda_nA_n(x_n-\tilde{x}_n)=(0,b_n),
\end{equation*}
where $0<b_n\leq C_0$. Up to a subsequence, $b_n\to b\geq 0$.

Let $y_0=(0,b)$, then
\begin{equation*}
|\sigma_v(y_0)|=\lim_{n\to \infty} |\sigma_{v_n}(y_n)|=1.
\end{equation*}
Therefore, we obtain a contradiction again.~\qed\bigskip

Now we can give the~\\
\noindent\textbf{Proof of \Cref{th5.1}.} We use the method of continuity to prove the existence. Since the fully nonlinear elliptic operator (see \cref{e4.2}) is not smooth, we first consider its mollification:
\begin{equation*}
  F_{\varepsilon}(p,M):=\int_{\mathcal{B}_1} F(p+\varepsilon q, M+\varepsilon N) \eta(q,N)~ dq dN,
\end{equation*}
where $\mathcal{B}_1$ is the unit ball (see \cref{e2.13} for the definition) and $\eta$ is a mollifier. That is, $\eta\in C^{\infty}$ is supported in $\mathcal{B}_1$ and $\int_{\mathcal{B}_1}\eta =1$. Then $F_{\varepsilon}\in C^{\infty}$.

Note that \Cref{le5.1} holds as well if we replace $F$ by $tF_{\varepsilon}+(1-t)\Delta$ (see \cite[p. 311]{MR1814364}), where $t\in [0,1]$ and $\Delta$ is the Laplace operator. Moreover, the constant $C$ in \cref{e3.0} has the same dependence and, in particular, is independent of $t$ and $\varepsilon$. In addition, it is also true for \Cref{le5.2}. We explain it in the following.

We adopt the same proof as in \Cref{le5.2}. In the procedure of blow-up, $v_n$ is a solution of the following equation (instead of \cref{e5.4}):
\begin{equation}\label{e3.3}
tF_{\varepsilon,n}(Dv_n, D^2v_n)+(1-t)\Delta v_n=0 \quad\mbox{in}~~B_{R_n},
\end{equation}
where
\begin{equation}\label{e3.7}
F_{\varepsilon,n}(p,M):=\lambda_n^{-1}F_{\varepsilon}(p,\lambda_n M).
\end{equation}
The key is that \cref{e3.3} is uniformly elliptic if $|Dv_n|$ and $|D^2v_n|$ are bounded. Indeed, by \cref{e2.12}, we know that $F_M$ and $F_p$ are $0$-homogenous and $1$-homogenous in $M$ respectively. Hence,
\begin{equation*}
\left|\frac{\partial F_{\varepsilon,n}}{\partial M}(p,M)\right|=\left|\int \frac{\partial F}{\partial M} (p+\varepsilon q, \lambda_n M+\varepsilon N) \eta(q,N)~ dq dN\right|\leq C,
\end{equation*}
where $C$ depends only on $|p|$ and is independent of $|M|$ and $\lambda_n$. In addition,
\begin{equation*}
\left|\frac{\partial F_{\varepsilon,n}}{\partial p}(p,M)\right|=\lambda_n^{-1}\left|\int \frac{\partial F}{\partial p} (p+\varepsilon q, \lambda_n M+\varepsilon N) \eta(q,N)~ dq dN\right|\leq C,
\end{equation*}
where $C$ depends only on $|p|$ and $|M|$. Therefore, \cref{e3.3} is uniformly elliptic if $|Dv_n|$ and $|D^2v_n|$ are bounded. The rest proof is the same as that of \Cref{le5.2}.

Next, we use the method of continuity to show the existence of solution of
\begin{equation}\label{e3.8}
  \left\{
  \begin{aligned}
    &F_{\varepsilon}(Du,D^2u)=0&&\quad\mbox{in}~~\Omega,\\
    &u=\varphi&&\quad\mbox{on}~~\partial \Omega.
  \end{aligned}
  \right.
\end{equation}
Consider the following Dirichlet problem:
\begin{equation}\label{e3.10}
  \left\{
  \begin{aligned}
    &tF_{\varepsilon}(Du,D^2u)+(1-t)\Delta u=0&&\quad\mbox{in}~~\Omega,\\
    &u=\varphi&&\quad\mbox{on}~~\partial \Omega.
  \end{aligned}
  \right.
\end{equation}
Let
\begin{equation*}
T=\left\{t\in [0,1]: \cref{e3.10} ~~\mbox{has a solution}~~ u_t\in C^{2,\beta}(\bar{\Omega})~~\mbox{for}~~t\right\},
\end{equation*}
where $\beta\leq \alpha$ is to be specified later. Clearly, $0\in T$. In the following, we show that $T$ is both open and closed.

We first show that $T$ is open, which is a consequence of the linear Schauder theory and the implicit function theorem. Suppose that $t_0\in T$. Then there exists a solution $u_0$ of \cref{e3.10}. Let
\begin{equation*}
C^{2,\beta}_0(\bar{\Omega}):=\left\{u\in C^{2,\beta}(\bar{\Omega}):~~u=0\quad\mbox{on}~~\partial \Omega\right\}.
\end{equation*}
Define an operator $\phi: C^{2,\beta}_0(\bar{\Omega})\times \mathbb{R}\to C^{\beta}(\bar{\Omega})$ as follows:
\begin{equation*}
\phi(u,t):=tF_{\varepsilon}(Du_0+Du,D^2u_0+D^2u)+(1-t)\Delta (u_0+u).
\end{equation*}
Clearly, $\phi(0,t_0)=0$. Now, we prove that $D_u\phi (0,t_0):C^{2,\beta}_0(\bar{\Omega})\to C^{\beta}(\bar{\Omega})$ is an invertible bounded linear operator. Indeed, by a direct calculation,
\begin{equation}\label{e3.11}
\begin{aligned}
\langle D_u\phi (0,t_0), w\rangle=&tF_{\varepsilon,M_{ij}}(Du_0,D^2u_0)w_{ij}+tF_{\varepsilon,p_{i}}(Du_0,D^2u_0)w_{i}\\
&+(1-t)\Delta w,~\forall ~w\in C^{2,\beta}_0(\bar{\Omega}).
\end{aligned}
\end{equation}
Let
\begin{equation*}
a^{ij}:= tF_{\varepsilon,M_{ij}}(Du_0,D^2u_0)+(1-t)\delta^{ij}, \quad
b^i=tF_{\varepsilon,p_{i}}(Du_0,D^2u_0).
\end{equation*}
Then \cref{e3.11} can be rewritten as
\begin{equation*}
\langle D_u\phi (0,t_0), w\rangle=a^{ij}w_{ij}+b^iw_i.
\end{equation*}

Since $F_{\varepsilon}$ is smooth and $u_0\in C^{2,\beta}(\bar{\Omega})$, we have $a^{ij},b^i\in C^{\beta}(\bar{\Omega})$. By the classical Schauder theory for linear uniformly elliptic equation (see \cite[Theorem 6.14]{MR1814364}), for any $h\in C^{\beta}(\bar{\Omega}), $ there exists a unique solution $w\in C^{2,\beta}
(\bar{\Omega})$ of
\begin{equation*}
\left\{\begin{aligned}
&a^{ij}w_{ij}+b^iw_i=h&& ~~\mbox{in}~~\Omega;\\
&w=0&& ~~\mbox{on}~~\partial \Omega
\end{aligned}\right.
\end{equation*}
and
\begin{equation*}
\|w\|_{C^{2,\beta}(\bar{\Omega})}\leq C\|h\|_{C^{\beta}(\bar \Omega)}.
\end{equation*}
Thus, $D_u\phi (0,t_0)$ is an invertible linear operator from $C^{2,\beta}_0(\bar{\Omega})$ to $C^{\beta}(\bar{\Omega})$. By the implicit function theorem, there exists a neighborhood $(t_0-\varepsilon,t_0+\varepsilon)$ of $t_0$ such that for any $t\in (t_0-\varepsilon,t_0+\varepsilon)$, there exists a unique $u\in C^{2,\beta}_0(\bar{\Omega})$ such that $\phi(u,t)=0$. In other words, $u_0+u$ is a solution of \cref{e3.10}. Therefore, we have proved that $T$ is open, i.e., if $t_0\in T$, there exists a neighborhood $(t_0-\varepsilon,t_0+\varepsilon)\cap [0,1]\subset T$.

Next, we prove that $T$ is closed. Given a sequence of $t_m\in T$ with $t_m\to \bar t$, we need to prove that $\bar t\in T$. By the definition of $T$, there exists a sequence of solutions $u_m$ of \cref{e3.10} with $t=t_m$ there. By \Cref{le5.1} and \Cref{le5.2}, we have
\begin{equation*}\label{e5.8}
\|u_m\|_{C^{2}(\bar{\Omega})}\leq C,
\end{equation*}
where $C$ is independent of $m$. Then the equation in \cref{e3.10} is uniformly elliptic with ellipticity constant independent of $m$. By the global $C^{2,\alpha}$ regularity (see \cite[Corollary 8.4]{lian2020pointwise}), $u_m\in C^{2,\beta}(\bar{\Omega})$ for some $0<\beta\leq \alpha$ and
\begin{equation*}
\|u_m\|_{C^{2,\beta}(\bar{\Omega})}\leq C.
\end{equation*}

Since $C^{2,\beta}(\bar{\Omega})$ is compactly embedded into $C^{2}(\bar{\Omega})$, there exists $\bar{u}\in C^{2,\beta}(\bar{\Omega})$ such that (up to a subsequence)
\begin{equation*}
u_m\to \bar{u}\quad\mbox{in}~~C^{2}(\bar{\Omega})
\end{equation*}
Take the limit in
\begin{equation*}
\left\{\begin{aligned}
&t_mF_{\varepsilon}(Du_m,D^2u_m)+(1-t_m)\Delta u_m=0&& ~~\mbox{in}~~\Omega;\\
&u_m=\varphi&& ~~\mbox{on}~~\partial \Omega.
\end{aligned}\right.
\end{equation*}
Then we have
\begin{equation*}
\left\{\begin{aligned}
&\bar tF_{\varepsilon}(D\bar{u}, D^2\bar{u})+(1-\bar t)\Delta \bar u=0&& ~~\mbox{in}~~\Omega;\\
&\bar u=\varphi&& ~~\mbox{on}~~\partial \Omega.
\end{aligned}\right.
\end{equation*}
That is, $\bar{u}$ is a solution and $\bar{t}\in T$.

From the above arguments, we have shown that $T$ is both open and closed. Therefore, $1\in A$. That is, there exists a solution of \cref{e3.8}.

We have proved that for any $0<\varepsilon<1$, there exists a solution $u_{\varepsilon}\in C^{2,\beta}(\bar{\Omega})$ of \cref{e3.8} and
\begin{equation*}
\|u_{\varepsilon}\|_{C^{2,\beta}(\bar{\Omega})}\leq C,
\end{equation*}
where $C$ is independent of $\varepsilon$. Then there exists $u\in C^{2,\beta}(\bar{\Omega})$ such that (up to a subsequence)
\begin{equation*}
u_{\varepsilon}\to u\quad\mbox{in}~~C^{2}(\bar{\Omega})
\end{equation*}
and
\begin{equation*}
\|u\|_{C^{2,\beta}(\bar{\Omega})}\leq C.
\end{equation*}

Take the limit in
\begin{equation*}
\left\{\begin{aligned}
&F_{\varepsilon}(Du_{\varepsilon},D^2u_{\varepsilon})=0&& ~~\mbox{in}~~\Omega;\\
&u_{\varepsilon}=\varphi&& ~~\mbox{on}~~\partial \Omega.
\end{aligned}\right.
\end{equation*}
Then we have
\begin{equation*}
\left\{\begin{aligned}
&F(Du, D^2u)=0&& ~~\mbox{in}~~\Omega;\\
&u=\varphi&& ~~\mbox{on}~~\partial \Omega.
\end{aligned}\right.
\end{equation*}
That is, $u$ is a solution of \cref{e5.1bis}.~\qed\bigskip

\section{The constant sign property}\label{S4}

We now consider an end of a uniformly elliptic Weingarten surface of minimal type $\Sigma$. We have previously shown that if this end has finite total curvature, then it may be represented as the graph of a function $u$ defined on the exterior of a ball $B_R \subset \mathbb{R}^2$, whose limit unit normal is $N_{\infty} = (0,0,1)$, or equivalently, such that $\lim_{x \to \infty} |Du| = 0$.

Our goal in this section is to prove \Cref{th1.1}; that is, we aim to show that, up to a vertical translation, either $u > 0$ or $u < 0$ in $B_R^c$, and that the limit
\[
u_{\infty} := \lim_{x \to \infty} u(x) \in \mathbb{R} \cup \{\pm\infty\}
\]
exists.

The proof is inspired by \cite[Theorem 1.11]{LZ_asym} (see also \cite{MR4201294}).\\[2mm]
\noindent\textbf{Proof of \Cref{th1.1}.} Using \Cref{th5.1} and noting $u\in C^{2,\alpha}$ (see \Cref{re2.2}), we can consider the unique solution $v_n$ to the problem
\begin{equation*}
  \left\{
  \begin{aligned}
    \kappa_2&=f(\kappa_1)\quad&&\mbox{in}~~B_n,\\
    v_n&=u\quad&&\mbox{on}~~\partial B_n.
  \end{aligned}
  \right.
\end{equation*}
Then there exist infinite $n$ such that
\begin{equation}\label{e7.3-2}
\max_{x\in \partial B_R}(v_n(x)-u(x))\geq 0
\end{equation}
or
\begin{equation}\label{e7.9-2}
\min_{x\in \partial B_R}(v_n(x)-u(x))\leq 0.
\end{equation}

Let us suppose that case \cref{e7.3-2} holds (case \cref{e7.9-2} can be treated similarly). Let $\Sigma_u$, $\Sigma_n$ be the graphs associated to $u,v_n$, respectively. By translating $\Sigma_n$ vertically downward if necessary, we may assume that
$$
\max_{x \in \partial B_R} (v_n(x) - u(x)) = 0
\quad \text{and so} \quad
\max_{x \in \partial B_n} (v_n(x) - u(x)) \leq 0.
$$
That is, $\Sigma_n$ is below $\Sigma_u$ at every point $x\in\partial B_R\cup\partial B_n$, and $\Sigma_n$ intersects $\Sigma_u$ at a certain point $(x_n,u(x_n))=(x_n,v_n(x_n))$ with $x_n\in\partial B_R$.

Thus, from the maximum principle,
\begin{equation}\label{vnun}
v_n(x)\leq u(x)\qquad\forall x\in\bar{B}_n-B_R.
\end{equation}

\Cref{le2.1} guarantees the existence of curvature estimates. Therefore, from \cite[Proposition 2.3]{MR2669367}, there exists a constant $\delta > 0$ such that, for every point $p \in \Sigma_n$ whose distance to the boundary is greater than $r > 0$, there is a neighborhood of $p$ in $\Sigma_n$ which can be represented as the graph of a function $w_p$ over the Euclidean disk $D_{\delta}$ of radius $\delta$, centered at $p$ in the tangent plane to $\Sigma_n$ at $p$. Moreover,
$$
\|w_p\|_{C^2(D_{\delta})}\leq\delta.
$$
Here, the positive numbers $\delta$ and $ r$ depend only on $\Lambda$, but not on $n$ or $p$.

By the interior $C^{2,\alpha}$ regularity (see \Cref{le2.2}), there exists $0<\beta<1$ such that $w_p\in C^{2,\beta}(\bar{B}_{\delta/2})$ and $\|w_p\|_{C^{2,\beta}(\bar{B}_{\delta/2})}$ is uniformly bounded. On the other hand, up to a subsequence, we may assume that the points $(x_n,u(x_n))\in\Sigma_u\cap\Sigma_n$ converge to some point $(x_0,u(x_0)) \in \Sigma_u$, with $x_0\in\partial B_R$, and that their Gauss maps satisfy $N_n(x_n,u(x_n)) \to N_0 \in \mathbb{S}^2$, where $N_n$ denotes the Gauss map of $\Sigma_n$.

From here, by the Arzela-Ascoli theorem,
the surfaces $\Sigma_n$ converge, up to a subsequence, in the $C^2$-norm on compact sets to a complete
surface $\Sigma_0$. But, from \cref{e1.3}, its second fundamental form vanishes identically, that is, $\Sigma_0$ must be a plane.

Therefore, from \cref{vnun}, the plane $\Sigma_0$ is below $\Sigma_u$. Moreover, observe that $\Sigma_0$ must be horizontal since $\lim_{x \to \infty} |Du| = 0$.
Thus, up to a vertical translation, we obtain \cref{e1.12}.

So, $u$ is a positive (or negative) solution of the uniformly elliptic linear equation \cref{e2.9}. Then, the   existence of $u_{\infty} := \lim_{x \to \infty} u(x) \in \mathbb{R} \cup \{\pm\infty\}$ can be now deduced (see, for instance, \cite[Lemma 2.1]{LZ_asym}) as a standard consequence of the interior Harnack inequality (see \cite[Corollary 9.25.]{MR1814364}).~\qed\bigskip

 \section{Proof of main results: the case \texorpdfstring{$f'(0)=-1$}{f'(0)=-1}}\label{S5}
In this section, we focus on the case where the relation between the principal curvatures is symmetric, i.e.,
$f'(0)=-1$. We derive the expansion \cref{e1.10} and establish \Cref{th1.2} and \Cref{co1.1}.

When $f'(0)=-1$, the corresponding fully nonlinear operator $F\in C^{1,1}$ is smooth (see \Cref{le2.5}). We aim to show that, in this situation, the asymptotic behavior of the solution coincides with that of harmonic functions. In fact, we prove a more general result that provides an expansion for general fully nonlinear uniformly elliptic equations with smooth operators.

 \begin{theorem}\label{th2.1}
Let $u$ be a solution of
\begin{equation*}\label{e3.4}
F(Du,D^2u)= 0\quad\mbox{in}~~\bar{B}_R^c,
\end{equation*}
where $F\in C^{1,1}$ is uniformly elliptic and $F(p,0)\equiv 0$. Suppose that
\begin{equation}\label{e2.25}
\|Du\|_{L^{\infty}(B_R^c)}\leq C, \quad |D^2u(x)|\leq C|x|^{-1},~\forall ~x\in B_R^c.
\end{equation}

Then there exist $b\in \mathbb{R}^2$, $c,d\in \mathbb{R}$ and a positively definite diagonal matrix $P$ such that
\begin{equation}\label{e1.4}
u(x)=b\cdot x+d\log |Px|+c+O(|x|^{-\alpha}),~\forall ~x\in \bar{B}_R^c, ~\forall ~0<\alpha<1
\end{equation}
and
\begin{equation*}\label{e1.4-2}
|Du(x)-b|= O(|x|^{-1}), \quad
|D^2u(x)|= O(|x|^{-2}),~\forall ~x\in \bar{B}_R^c.
\end{equation*}
 \end{theorem}
\proof Since $F(p,0)\equiv 0$, $u$ is a solution of the following linear uniformly elliptic equation:
\begin{equation}\label{e3.9}
a^{ij}u_{ij}=0\quad\mbox{in}~~ \bar{B}_{2R}^c,
\end{equation}
where
\begin{equation*}
a^{ij}(x):=\int_{0}^{1}F_{M_{ij}}(Du(x),tD^2u(x)) dt.
\end{equation*}
Then by the same proof as in \Cref{le2.3}, with the aid of \cite[Theorem 2.6]{MR4792754}, there exist $0<\alpha<1$ and $b\in \mathbb{R}^2$ such that
\begin{equation}\label{e6.13}
|Du(x)-b|\leq C|x|^{-\alpha}, \quad
|D^2u(x)|\leq  C|x|^{-1-\alpha},~\forall ~x\in \bar{B}_{2R}^c.
\end{equation}

Now the proof follows the approach of \cite{MR4792754}. The equation \cref{e3.9} can be rewritten as
\begin{equation*}
F_{M_{ij}}(b,0)u_{ij}=\left(F_{M_{ij}}(b,0)-a^{ij}\right)u_{ij} \quad\mbox{in}~~ \bar{B}_{2R}^c.
\end{equation*}
Note that $F_{M_{ij}}(b,0)$ is a positively definite matrix. Up to a change of variables $y=PQx$ for some positively definite diagonal matrix $P$ and some orthogonal matrix $Q$, we assume that $F_{M_{ij}}(b,0)$ is the unit matrix $\delta^{ij}$ and then $u$ is a solution of
\begin{equation}\label{e6.10}
\Delta u=f:=\left(\delta^{ij}-a^{ij}\right)u_{ij}\quad\mbox{in}~~ \bar{B}_{2R}^c.
\end{equation}

By $F\in C^{1,1}$ and \cref{e6.13},
\begin{equation}\label{e6.11}
\begin{aligned}
|a^{ij}(x)-\delta^{ij}|=&\int_{0}^{1}\left(F_{M_{ij}}(Du(x),tD^2u(x)) -F_{M_{ij}}(b,0)\right)dt\\
\leq& C\left(|x|^{-\alpha}+|x|^{-1-\alpha}\right)\leq C|x|^{-\alpha},~\forall ~x\in \bar{B}_{2R}^c.
\end{aligned}
\end{equation}

Then by \cref{e6.13} and \cref{e6.11},
\begin{equation*}
|f(x)|\leq C|x|^{-1-2\alpha},~\forall ~x\in \bar{B}_{2R}^c.
\end{equation*}
For $\varepsilon>0$ small enough, let $v$ be a solution of (see \cite[Lemma 3.4]{MR4792754} and a detailed proof can be found in \cite[p. 46]{MR3299174})
\begin{equation*}
\Delta v=f\quad\mbox{in}~~ \bar{B}_{2R}^c
\end{equation*}
with
\begin{equation*}
\begin{aligned}
&|v(x)|+|x||Dv(x)|+|x|^2|D^2v(x)|\leq C|x|^{1-2\alpha+\varepsilon},~~\forall ~x\in \bar{B}_{2R}^c.
\end{aligned}
\end{equation*}

Hence, $u-v$ is a harmonic function and from \cite[Lemma 3.5]{MR4792754}, there exist constants $c,d$ such that
\begin{equation}\label{e6.12}
u(x)-v(x)= b\cdot x + d \log |x| + c + O(|x|^{-1})\quad\mbox{as}~~|x|\to \infty.
\end{equation}

Let
\begin{equation*}
\bar{u}=u- b\cdot x
\end{equation*}
and then
\begin{equation}\label{e6.2}
\bar{u}(x)=v(x)+d \log |x| + c + O(|x|^{-1})
=O(|x|^{1-2\alpha+\varepsilon}),~\forall ~x\in \bar{B}_{2R}^c.
\end{equation}
In addition, $\bar{u}$ is a solution of
\begin{equation*}\label{e6.1}
F(D\bar{u}+b,D^2\bar{u})=0\quad\mbox{in}~~\bar{B}_{2R}^c.
\end{equation*}

Similar to the proof of \Cref{le2.3}, by \cref{e6.2} and the interior $C^{2,\alpha}$ regularity (see \Cref{le2.2}) for $\bar{u}$, we have
\begin{equation*}
\begin{aligned}
&|D\bar{u}(x)|\leq C|x|^{-2\alpha+\varepsilon}, \quad
&|D^2\bar{u}(x)|\leq C|x|^{-1-2\alpha+\varepsilon},~\forall ~x\in \bar{B}_{2R}^c.
\end{aligned}
\end{equation*}
Therefore,
\begin{equation*}\label{e6.3}
|Du(x)-b|\leq C|x|^{-2\alpha+\varepsilon}, \quad
|D^2u(x)|\leq  C|x|^{-1-2\alpha+\varepsilon},~\forall ~x\in \bar{B}_{2R}^c.
\end{equation*}

Then similar to \cref{e6.11}, we have
\begin{equation*}\label{e6.4}
\begin{aligned}
|a^{ij}(x)-\delta^{ij}|=&\int_{0}^{1}\left(F_{M_{ij}}(Du(x),tD^2u(x)) -F_{M_{ij}}(b,0)\right)dt\\
\leq&\ C\left(|x|^{-2\alpha+\varepsilon}+|x|^{-1-2\alpha+\varepsilon}\right)\leq C|x|^{-2\alpha+\varepsilon},~\forall ~x\in \bar{B}_{2R}^c.
\end{aligned}
\end{equation*}
By repeating the argument from \cref{e6.10} to \cref{e6.12}, we obtain \cref{e6.12} with
\begin{equation*}
v(x)=O(|x|^{1-4\alpha+2\varepsilon}),~\forall ~x\in \bar{B}_{2R}^c.
\end{equation*}
Repeat the above process. After finitely many iterations, we obtain \cref{e6.12} with
\begin{equation*}
v(x)=O(|x|^{-\beta}),~\forall ~x\in \bar{B}_{2R}^c
\end{equation*}
for some $0<\beta<1$. That is, we have the following expansion for $u$ at infinity
\begin{equation*}\label{e6.14}
u(x)= b\cdot x + d \log |x| + c + O(|x|^{-\beta})\quad\mbox{as}~~|x|\to \infty.
\end{equation*}

Finally, let
\begin{equation*}
\bar{u}=u- b\cdot x
\end{equation*}
again and then
\begin{equation*}
\bar{u}(x)=O(\log |x|),~\forall ~x\in \bar{B}_{2R}^c.
\end{equation*}
By repeating the above process once again, we obtain \cref{e6.12} with
\begin{equation}\label{e5.9}
v(x)=O(|x|^{-\gamma}), \quad |Dv(x)|=O(|x|^{-1-\gamma}), \quad
|D^2v(x)|=O(|x|^{-2-\gamma}), \quad~\forall ~x\in \bar{B}_{2R}^c
\end{equation}
for any $0<\gamma<1$. Therefore, we have the following expansion for $u$ at infinity
\begin{equation*}
u(x)= b\cdot x + d \log |x| + c + O(|x|^{-\gamma})\quad\mbox{as}~~|x|\to \infty.
\end{equation*}

Let
\begin{equation*}
w=u-v-b\cdot x - d \log |x| -c,
\end{equation*}
which is harmonic in $\bar B_{2R}^c$. From \cref{e6.12}, $w(x)=O(|x|^{-1})$ and hence
\begin{equation*}
|Dw(x)|=O(|x|^{-2}), \quad |D^2w(x)|=O(|x|^{-3}). 
\end{equation*}
Therefore, by noting \cref{e5.9}, we have
\begin{equation*}
\begin{aligned}
&|Du(x)-b|= O(|x|^{-1}), \quad
|D^2u(x)|= O(|x|^{-2}),~\forall ~x\in \bar{B}_{2R}^c.
\end{aligned}
\end{equation*}
~\qed\bigskip

Now, we prove \Cref{th1.3} for the case $f'(0)=-1$.\\[2mm]
\noindent\textbf{Proof of \Cref{th1.3} (the case $f'(0)=-1$).} By \Cref{le2.6}, $u$ is a solution of a fully nonlinear uniformly elliptic equation in the form \cref{e2.3}. Since our graph is uniformly elliptic of minimal type, we have $F(p,0)\equiv 0$ (see \cref{e4.2}). From \Cref{le2.5}, $F\in C^{1,1}$. Moreover, \cref{e2.25} holds by \Cref{le2.3}.

Therefore, the assumptions of \Cref{th2.1} are all satisfied and we have the expansion \cref{e1.4}. Note that $b=0$ since the limit unit normal of the graph is $N_{\infty}=(0,0,1)$. In addition, $F_{M_{ij}}(0,0)=\delta^{ij}$ (see \cref{e2.12} and note $f'(0)=-1$), so we obtain the expansion \cref{e1.10}.~\qed\bigskip

Next, we prove the strong comparison principle at infinity.~\\
\noindent\textbf{Proof of \Cref{th1.2}.} By \Cref{th1.3}, there exist $c,\tilde{c},d,\tilde{d}\in \mathbb{R}$ such that for any $0<\alpha<1$,
\begin{equation}\label{e2.18-0}
\begin{aligned}
&u(x)=d\log |x|+c+O(|x|^{-\alpha}),~\forall ~x\in \bar{B}_R^c,\\
&\tilde u(x)=\tilde d\log |x|+\tilde c+O(|x|^{-\alpha}),~\forall ~x\in \bar{B}_R^c
\end{aligned}
\end{equation}
and
\begin{equation}\label{e2.18}
\begin{aligned}
&|Du(x)|= O(|x|^{-1}), \quad |D^2u(x)|= O(|x|^{-2}),~\forall ~x\in \bar{B}_R^c,\\
&|D\tilde u(x)|= O(|x|^{-1}), \quad |D^2\tilde u(x)|= O(|x|^{-2}),~\forall ~x\in \bar{B}_R^c.
\end{aligned}
\end{equation}

Let $v=u-\tilde{u}$. Then $v$ is a solution of the following linear uniformly elliptic equation:
\begin{equation}\label{e2.1}
a^{ij}v_{ij}+a^iv_i=0\quad\mbox{in}~~ \bar{B}_R^c,
\end{equation}
where
\begin{equation*}
\begin{aligned}
&a^{ij}(x):=\int_{0}^{1}F_{M_{ij}}(\xi) dt, \quad a^i(x):=\int_{0}^{1}F_{p_{i}}(\xi) dt,\\
&\xi:=(\xi_1,\xi_2):=(D\tilde u(x)+t(Du(x)-D\tilde{u}(x)),D^2\tilde u(x)+t(D^2u(x)-D^2\tilde{u}(x))).
\end{aligned}
\end{equation*}
As the calculation in \cref{e6.11}, by noting \cref{e2.18}, we have
\begin{equation}\label{e2.22}
|a^{ij}(x)-\delta^{ij}|\leq C|x|^{-1},~\forall ~x\in \bar{B}_R^c.
\end{equation}
Similarly, by noting $F\in C^{1,1}$, $F_{p}(\xi_1,0)\equiv 0$ and \cref{e2.18}, we have
\begin{equation*}
|F_{p}(\xi)|=|F_{p}(\xi)-F_p(\xi_1,0)|\leq C|\xi_2|\leq C\left(|D^2u(x)|+|D^2\tilde u(x)|\right)\leq C|x|^{-2}.
\end{equation*}
Thus,
\begin{equation}\label{e2.19}
|a^i|=\left|\int_{0}^{1}F_{p_{i}}(\xi) dt\right|= O(|x|^{-2}),~\forall ~x\in \bar{B}_R^c.
\end{equation}

Once we have the estimate \cref{e2.19} for $a^i$, the limit of $v(x)$ exists as $x\to\infty$ (see \cite[Theorem 3]{MR81416}). A standard modern proof relies on the interior Harnack inequality (see \cite[Corollary 9.25]{MR1814364}); see \cite[Lemma 2.1]{LZ_asym} or \cite[Theorem 2.2]{MR4038557} for this standard proof (concerning equations without lower order terms). Thus,
\begin{equation}\label{e2.20}
\lim_{x\to \infty} v(x)=+\infty\quad\mbox{or}~~\lim_{x\to \infty} v(x)=c_0\geq 0.
\end{equation}
If the first case holds, then $d>\tilde{d}$ in \cref{e2.18-0}. Hence, we have
\cref{e1.9}.

In the following we consider the second case in \cref{e2.20}. Consider the positive function
\begin{equation*}
  w(x)=\log \log |x|, \quad x\in B_{3}^c.
\end{equation*}
The key is to show that $w$ is a supersolution of \cref{e2.1} in $B_{r_0}^c$ for some $r_0$ large enough, which is achieved by some calculations. First, by direct calculation,
\begin{equation*}\label{e2.21}
\begin{aligned}
&|Dw(x)|\leq \left(\log |x|\right)^{-1}|x|^{-1}, \quad |D^2w(x)|\leq C\left(\log |x|\right)^{-1}|x|^{-2}, \\
&\Delta w(x)=-\left(\log |x|\right)^{-2}|x|^{-2},~\forall ~x\in B_{3}^c.
\end{aligned}
\end{equation*}
Then, using \cref{e2.22} and \cref{e2.19}, we have
\begin{equation*}
  \begin{aligned}
a^{ij}w_{ij}+a^iw_i=&\Delta w+(a^{ij}-\delta^{ij})w_{ij}+a^iw_i\\
\leq& -\left(\log |x|\right)^{-2}|x|^{-2}+C\left(\log |x|\right)^{-1}|x|^{-3}.
  \end{aligned}
\end{equation*}
Hence, there exists $r_0$ large enough such that
\begin{equation*}\label{e2.23}
a^{ij}w_{ij}+a^iw_i\leq 0\quad\mbox{in}~~\bar B_{r_0}^c.
\end{equation*}
That is, $w$ is a supersolution.

Since $v$ is bounded and $w(x)\to +\infty$ as $x\to \infty$, then for any $\varepsilon>0$, there exists $r_1>r_0$ such that for any $r\geq r_1$,
\begin{equation*}
v\leq \max_{\partial B_{r_0}}v+\varepsilon w\quad\mbox{on}~~\partial B_{r}.
\end{equation*}
By the comparison principle (note that $w>0$),
\begin{equation*}
v\leq \max_{\partial B_{r_0}} v+\varepsilon w\quad\mbox{in}~~B_{r}\backslash B_{r_0}.
\end{equation*}
Let $r\rightarrow \infty$ first and then $\varepsilon\rightarrow 0$, we have
\begin{equation*}
  v\leq \max_{\partial B_{r_0}}v \quad \mbox{in}~~B_{r_0}^c.
\end{equation*}
The inequality
\begin{equation*}
 v\geq \min_{\partial B_{r_0}}v \quad \mbox{in}~~B_{r_0}^c
\end{equation*}
can be proved in a similar way and therefore
\begin{equation*}
\min_{\partial B_{r_0}}v\leq v(x)\leq \max_{\partial B_{r_0}}v, ~\forall ~x\in B_{r_0}^c.
\end{equation*}
Thus,
\begin{equation*}
\min_{\partial B_{r_0}}v\leq c_0 \leq \max_{\partial B_{r_0}}v.
\end{equation*}
Similarly, this can be proved for any $r>r_0$, that is, we have
\begin{equation*}
\min_{\partial B_{r}}v\leq c_0\leq \max_{\partial B_{r}}v, ~\forall ~r>r_0.
\end{equation*}
Hence, \cref{e1.6-3} holds. ~\qed\bigskip

We are now ready to prove \Cref{co1.1}.~\\
\noindent\textbf{Proof of \Cref{co1.1}.} This is a direct consequence of \Cref{th1.2}. Indeed, by the assumption and \Cref{th1.2},
\begin{equation*}
u(x)-\tilde u(x)\to 0\quad\mbox{as}~~x\to \infty
\end{equation*}
and there exists $r_0>0$ such that
\begin{equation*}
0\leq \min_{\partial B_r} (u-\tilde u)\leq 0\leq \max_{\partial B_r} (u-\tilde u),~\forall ~ r>r_0.
\end{equation*}
Hence, $u(x)\geq \tilde u (x)$ in $B_R^c$ and there exists $x_0\in \bar{B}_R^c$ such that
\begin{equation*}
u(x_0)=\tilde{u}(x_0).
\end{equation*}
Then, by the strong maximum principle, $u\equiv \tilde{u}$.~\qed\bigskip

\section{Proof of main results: the case \texorpdfstring{$f'(0)\neq -1$}{f'(0)≠-1}}\label{S6}
In the case $f'(0)\neq -1$, the fully nonlinear operator $F$ lacks sufficient smoothness to yield a refined asymptotic expansion such as \cref{e1.10}. More precisely, under this assumption, \cref{e2.30} no longer holds. Consequently, the equation cannot be regarded as a perturbation of the Laplace equation at infinity. Instead, we directly compare the solution $u$ with a suitably chosen model solution $v$. Specifically, the model solution $v$ denotes the positive (or negative) radial solution of the same equation. Since $v$ is radial, the PDE reduces to an ODE, which is comparatively easier to analyze.

Local and global properties of radial solutions of \cref{e1.6} have been studied in \cite{MR4557820}. Moreover, explicit formulas for the rotational solutions of \cref{e1.6} in the particular case $f(x)=a x$, with $a\in\mathbb{R}$, were obtained in \cite{MR4089078}.

The following lemma constitutes the main technical tool used in the proof of \Cref{th1.3} for the case $f'(0)\neq -1$.

\begin{lemma}\label{le8.1}
Let $f\in C^{1,1}_{loc}[0,+\infty)$ and satisfy \cref{e1.6-2}--\cref{e1.6-t}. For any $R_0>0$ and $C_0> 0$, there exists a unique radial solution $u$ of
\begin{equation}\label{e8.0}
\kappa_2=f(\kappa_1)\quad\mbox{in}~~B_{R_0}^c
\end{equation}
with the initial conditions $u(R_0)=0$ and $u'(R_0)=C_0$. In addition, $u$ has the following properties:~\\[1mm]
(i) The solution $u$ is nondecreasing, $u'$ is nonincreasing and $u'(r)\to 0$ as $r\to \infty$. Moreover,
\begin{equation*}\label{e8.0-2}
0\leq u(r)\leq \frac{C_0R_0^{\Lambda}}{1-\Lambda}r^{1-\Lambda}, \quad  0<u'(r)\leq C_0, \quad 
-\frac{C_0(1+C_0^2)}{\Lambda R_0}\leq u''(r)<0,
~\forall ~r\in [R_0,+\infty).
\end{equation*}
(ii) If $-1<f'(0)<0$, there exists $K> 0$ such that
\begin{equation}\label{e8.1}
\lim_{r\to \infty} \frac{u(r)}{r^{1+f'(0)}}=K.
\end{equation}
Moreover, $K$ is a continuous function of $R_0,C_0$ and
\begin{equation}\label{e8.2-1}
K\to 0 ~~\mbox{ as }~~R_0\to 0 ~~(\mbox{or }~~C_0\to 0)\quad\mbox{and}\quad K\to \infty ~~\mbox{ as }~~R_0\to \infty.
\end{equation}
(iii) If $f'(0)<-1$, there exists $K> 0$ such that
\begin{equation*}\label{e8.1-2}
\lim_{r\to \infty} u(r)=K.
\end{equation*}
Moreover, $K$ is a continuous function of $R_0,C_0$ and
\begin{equation}\label{e8.2-2}
K\to 0 ~~\mbox{ as }~~R_0\to 0 ~~(\mbox{or }~~C_0\to 0)\quad\mbox{and}\quad K\to \infty ~~\mbox{ as }~~R_0\to \infty.
\end{equation}

\begin{figure}[h]
  \centering
  \includegraphics[width=0.7\textwidth]{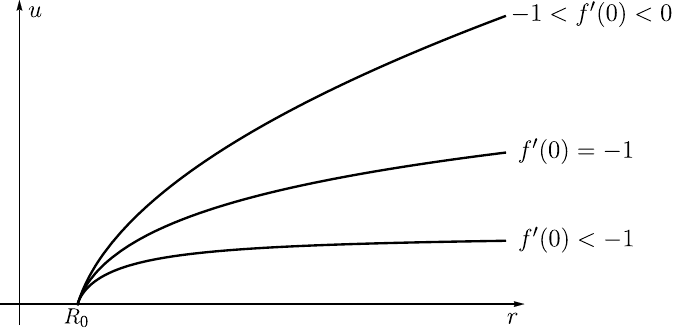}
  \caption{Asymptotic behavior of radial solutions $u(r)$ for different ranges of $f'(0)$.}\label{fig1}
\end{figure}

\end{lemma}
\proof Since we seek a radial solution, \cref{e8.0} reduces to
\begin{equation}\label{e8.3}
\frac{u''(r)}{\left(1+u'^2(r)\right)^{3/2}}=f\left(\frac{u'(r)}{r\left(1+u'^2(r)\right)^{1/2}}\right)\quad\mbox{for}~~ r>R_0.
\end{equation}
From the standard ODE theory (e.g., \cite[Theorem 1.1, Chapter 2]{MR1929104}), there exists a unique solution $u\in C^{3,1}$ to \cref{e8.3} in some interval $[R_0, R]$.

Now, we show that the interval of existence can be extended to $[R_0,+\infty)$. Let
\begin{equation*}
A=\left\{r\in (R_0,+\infty): u ~~\mbox{is a solution of \cref{e8.3} in}~~[R_0,r]~~
\mbox{and}~~u'>0~~\mbox{in}~~[R_0,r]\right\}.
\end{equation*}
Since $f$ is only defined in $[0,+\infty)$, we require $u'>0$ in the above definition.

Clearly, $A$ is a nonempty open subset of $[R_0,+\infty)$. We only need to show that $A$ is closed as well. Let $\left\{r_n\right\}\subset A$ with $r_n\to r_0$, and $r_n<r_0$ for any $n$. Then $u$ is a solution of \cref{e8.3} in $[R_0,r_0)$ and $u'> 0$. Using that $f(0)=0$, we can rewrite \cref{e8.3} as
\begin{equation}\label{e8.4}
u''=\frac{u'g}{r}\quad\mbox{in}~~[R_0,r_0),
\end{equation}
where
\begin{equation}\label{e8.4-2}
g(r)=\left(1+u'^2(r)\right)\int_{0}^{1} f'\left(\rho \frac{u'(r)}{r\sqrt{1+u'^2(r)}}\right)d\rho<0.
\end{equation}

Since $u(R_0)=0$ and $u'(R_0)=C_0$, we have
\begin{equation}\label{e8.5}
\begin{aligned}
u'(r)=C_0\exp\left(\int_{R_0}^{r}\frac{g(t)}{t}dt\right), \quad
u(r)=C_0\int_{R_0}^{r}\exp\left(\int_{R_0}^{s}\frac{g(t)}{t}dt\right)ds,\quad r\in [R_0,r_0).
\end{aligned}
\end{equation}

As $g<0$, we have $u''<0$. That is, $u'$ is decreasing and hence
\begin{equation}\label{e6.5}
0<u'\leq C_0\quad\mbox{in}~~[R_0,r_0).
\end{equation}
By combining with the uniformly elliptic condition \cref{e1.6-2} for $f$, we have
\begin{equation}\label{e8.6}
 -\Lambda^{-1}\left(1+C_0^2\right)\leq g\leq -\Lambda\quad\mbox{in}~~[R_0,r_0).
\end{equation}
Then
\begin{equation}\label{e6.8}
\int_{R_0}^{s}\frac{g(t)}{t}dt\leq -\Lambda \log \frac{s}{R_0},~\forall ~s\in [R_0,r_0)
\end{equation}
and hence
\begin{equation}\label{e6.9}
u(r)\leq \frac{C_0R_0^{\Lambda}}{1-\Lambda} r^{1-\Lambda},~\forall ~r\in [R_0,r_0).
\end{equation}
Moreover, by \cref{e8.4}, \cref{e6.5} and \cref{e8.6},
\begin{equation}\label{e8.13}
u''(r)=\frac{u'(r)g(r)}{r}\geq -\frac{C_0(1+C_0^2)}{\Lambda R_0},~\forall ~r\in [R_0,r_0).
\end{equation}

With the aid of \cref{e8.6}, we have
\begin{equation*}
C_1:=\lim_{r\to r_0}u'(r)=C_0\exp\left(\int_{R_0}^{r_0}\frac{g(t)}{t}dt\right)\geq C_0\left(\frac{r_0}{R_0}\right)^{-\Lambda^{-1}(1+C_0^2)}>0.
\end{equation*}
Similarly, the limit $\lim_{r\to r_0} u(r)$ exists as well. Thus, we can obtain a unique solution $v$ of \cref{e8.3} in a neighborhood $(r_0-\delta,r_0+\delta)$ ($\delta>0)$ subjected to the following initial conditions
\begin{equation*}
v(r_0)=\lim_{r\to r_0} u(r), \quad v'(r_0)=C_1.
\end{equation*}
By the uniqueness,
\begin{equation*}
v(r)=C_0\int_{R_0}^{r}\exp\left(\int_{R_0}^{s}\frac{g(t)}{t}dt\right)ds,\quad r\in (r_0-\delta,r_0).
\end{equation*}
That is, $u\equiv v$ in $(r_0-\delta,r_0)$. Hence, $u$ can be extended to $[R_0, r_0+\delta)$. Thus, $r_0\in A$ and $A$ is closed. This shows that $A=[R_0,+\infty)$.

In conclusion, $u$ can be extended to the entire interval $[R_0, +\infty)$ and satisfies all the equations from \cref{e8.4} to \cref{e8.13}, which proves (i).

Next, we consider the case $-1<f'(0)<0$. We need to prove that there exists a positive constant $K$ such that \cref{e8.1} holds. By the derivatives estimates (see \Cref{le2.3}), there exists $\alpha>0$ such that
\begin{equation}\label{e8.7}
0<u'(r)\leq Cr^{-\alpha}, \quad -Cr^{-1-\alpha}\leq u''(r)<0,~\forall ~r>R_0.
\end{equation}
To prove \cref{e8.1}, the key is the following observation (with the aid of \cref{e8.7}):
\begin{equation}\label{e8.8}
\begin{aligned}
|g(r)-f'(0)|\leq& u'^2(r)\int_{0}^{1} |f'|d\rho
+\int_{0}^{1}\left|f'\left(\rho \frac{u'(r)}{r\sqrt{1+u'^2(r)}}\right)-f'(0)\right|d\rho\\
\leq & u'^2(r)\Lambda^{-1}+\max_{s\in [0,1/R_0]}|f''(s)|\cdot \frac{u'(r)}{r\sqrt{1+u'^2(r)}}\\
\leq &Cr^{-2\alpha}+Cr^{-1-\alpha},~\forall ~r>R_0.
\end{aligned}
\end{equation}
Now, we can prove \cref{e8.1}. By the L'H\^{o}pital's rule,
\begin{equation}\label{e8.9}
  \begin{aligned}
K:=\lim_{r\to \infty} \frac{u(r)}{r^{1+f'(0)}}=&\lim_{r\to \infty} \frac{u'(r)}{(1+f'(0))r^{f'(0)}}\\
=&\lim_{r\to \infty} \frac{C_0}{1+f'(0)}\exp\left(-f'(0)\log r+\int_{R_0}^{r}\frac{g(t)}{t}dt\right)\\
=&\lim_{r\to \infty} \frac{C_0}{1+f'(0)}\exp\left(\int_{R_0}^{r}\frac{g(t)-f'(0)}{t}dt-f'(0)\log R_0\right)\\
=&\lim_{r\to \infty} \frac{C_0R_0^{-f'(0)}}{1+f'(0)}\exp\left(\int_{R_0}^{r}\frac{g(t)-f'(0)}{t}dt\right)\\
=&\frac{C_0R_0^{-f'(0)}}{1+f'(0)}\exp\left(\int_{R_0}^{\infty}\frac{g(t)-f'(0)}{t}dt\right),
  \end{aligned}
\end{equation}
where we have used \cref{e8.8} in the last step.

From the ODE theory (see \cite[Corollary 3.3, Chapter 5]{MR1929104}), the solution $u$ and its derivative $u'$ depend on the initial conditions $R_0,C_0$ in a continuous way. By combining with the expression of $u'$ (see \cref{e8.5}), we conclude that $\int_{R_0}^{r}g(t)/t\,dt$ depends on $R_0,C_0$ continuously. In addition, by noting the estimate \cref{e8.8}, $\int_{R_0}^{\infty}(g(t)-f'(0))/t\,dt$ depends on $R_0,C_0$ continuously. That is, $K$ depends on $R_0,C_0$ continuously.

Note that $\int_{R_0}^{\infty}(g(t)-f'(0))/t\,dt \to 0$ as $R_0\to \infty$. Hence, by the last line in \cref{e8.9} and $f'(0)<0$, we conclude that
\begin{equation*}
K\to \infty \quad\mbox{as }~~R_0\to \infty.
\end{equation*}
Next, we consider the situation $R_0\to 0$. By a direct calculation (cf. \cref{e8.8}),
\begin{equation*}
g(r)-f'(0)\leq \max_{s\in [0,1]}|f''(s)| r^{-1-\alpha},~\forall ~r\geq 1.
\end{equation*}
In addition, by combining with $g\leq -\Lambda$ (see \cref{e8.6}), we have
\begin{equation*}
\begin{aligned}
\int_{R_0}^{\infty}\frac{g(t)-f'(0)}{t}dt=\int_{R_0}^{1}\frac{g(t)-f'(0)}{t}dt+\int_{1}^{\infty}\frac{g(t)-f'(0)}{t}dt
\leq  (\Lambda+f'(0))\log R_0+\max_{s\in [0,1]}|f''(s)|.
\end{aligned}
\end{equation*}
Thus,
\begin{equation*}
K\leq \frac{C_0R_0^{\Lambda}}{1+f'(0)}\exp(\max_{s\in [0,1]}|f''(s)|)\to 0 \quad\mbox{as }~~R_0\to 0.
\end{equation*}
Similarly, we also have
\begin{equation*}
K\to 0\quad\mbox{as }~~C_0\to 0.
\end{equation*}
In summary, we have proved (ii). Note that we cannot obtain $K\to \infty$ as $C_0\to \infty$ since $g$ depends on $C_0$ in an implicit way.

Finally, we consider the case $f'(0)<-1$. From the expression of $g$ (see \cref{e8.4-2}) and the continuity of $f'$ at $0$, there exist $\varepsilon>0$ and $R_1>\max(R_0,1)$ (depending only on $f'(0)$ and the continuity of $f'$ at $0$) such that
\begin{equation}\label{e6.7}
g(r)<-1-\varepsilon,~\forall ~r\geq R_1.
\end{equation}
Hence,
\begin{equation}\label{e8.10}
\exp\left(\int_{R_0}^{s}\frac{g(t)}{t}dt\right)\leq Cs^{-1-\varepsilon},~\forall ~s>R_1,
\end{equation}
where $C$ depends on $R_0$ and is independent of $C_0$. Thus,
\begin{equation}\label{e8.11}
K:=\lim_{r\to \infty} u(r)=C_0\int_{R_0}^{\infty}\exp\left(\int_{R_0}^{s}\frac{g(t)}{t}dt\right)ds<+\infty.
\end{equation}

Clearly, $K\to 0$ as $C_0\to 0$. In addition, with the aid of \cref{e8.6},
\begin{equation*}
\begin{aligned}
K=C_0\int_{R_0}^{\infty}\exp\left(\int_{R_0}^{s}\frac{g(t)}{t}dt\right)ds
\geq C_0\int_{R_0}^{\infty}\left(s^{-\Lambda^{-1}(1+C_0^2)}R_0^{\Lambda^{-1}(1+C_0^2)}\right)ds
=\frac{C_0R_0}{\Lambda^{-1}(1+C_0^2)-1}.\\
\end{aligned}
\end{equation*}
Thus, $K\to \infty$ as $R_0\to \infty$. Finally, we consider $R_0\to 0$. Note that $g\leq -\Lambda$ (see \cref{e8.6}). Hence, when $s<R_1$,
\begin{equation}\label{e8.12}
\exp\left(\int_{R_0}^{s}\frac{g(t)}{t}dt\right)\leq  R_0^{\Lambda}s^{-\Lambda}.
\end{equation}
By combining \cref{e6.7} with \cref{e8.12}, we have
\begin{equation*}
  \begin{aligned}
K=&C_0\int_{R_0}^{R_1}\exp\left(\int_{R_0}^{s}\frac{g(t)}{t}dt\right)ds
+C_0\int_{R_1}^{\infty}\exp\left(\int_{R_0}^{s}\frac{g(t)}{t}dt\right)ds\\
\leq & C_0 R_0^{\Lambda}\int_{R_0}^{R_1}s^{-\Lambda}ds
+C_0\int_{R_1}^{\infty}\exp\left(\int_{R_0}^{R_1}\frac{g(t)}{t}dt
+\int_{R_1}^{s}\frac{g(t)}{t}dt\right)ds\\
\leq & CR_0^{\Lambda}+C\int_{R_1}^{\infty}R_0^{\Lambda}s^{-1-\varepsilon}ds \leq  CR_0^{\Lambda}.
  \end{aligned}
\end{equation*}
Therefore, $K\to 0$ as $R_0\to 0$. In conclusion, we have proved (iii). ~\qed\bigskip

Now, we can prove the main theorem. We divide it into two cases: $-1<f'(0)<0$ and $f'(0)<-1$. We first consider the former.
\begin{lemma}\label{le3.2}
Let $u$ define a positive solution to \cref{e1.6}--\cref{e1.6-2}--\cref{e1.6-t} in $B_R^c$ with  limit unit normal $N_{\infty}=(0,0,1)$ and $-1<f'(0)<0$. Then
\begin{equation*}
c:=\liminf_{x\rightarrow \infty} \frac{u(x)}{|x|^{1+f'(0)}}<+\infty.
\end{equation*}
Moreover,~\\
(i) If $c=0$, there exists a constant $K$ such that
\begin{equation}\label{e3.14}
\lim_{x\rightarrow \infty} u(x)=K.
\end{equation}
(ii) If $c>0$,
\begin{equation*}\label{e3.13}
\lim_{x\rightarrow \infty}\frac{u(x)}{|x|^{1+f'(0)}}=c.
\end{equation*}
\end{lemma}
\proof Suppose that $c=+\infty$, i.e.,
\begin{equation*}
\lim_{x\to \infty} \frac{|x|^{1+f'(0)}}{u(x)}=0.
\end{equation*}
Consider the function $u_{\delta}$ which defines the new graph obtained by scaling with factor $0<\delta<1$:
\begin{equation*}
u_{\delta}(x):=\delta\, u\left(\frac{x}{\delta}\right),~\forall ~x\in B_R^c.
\end{equation*}
Then $u_{\delta}$ is a solution of
\begin{equation}\label{e3.15}
\kappa_2=f_{\delta}(\kappa_1)\quad\mbox{in}~~B_R^c,
\end{equation}
where
\begin{equation*}
f_{\delta}(t):=\delta f(\delta^{-1}t),~\forall ~t\geq 0.
\end{equation*}

With respect to this $f_{\delta}$, by \Cref{le8.1}, there exists a radial solution $v_{\delta}$ of \cref{e3.15} with $v_{\delta}(R)=0$ and $v'_{\delta}(R)=1$. Moreover (note that $f'_{\delta}(0)=f'(0)$),
\begin{equation*}
\lim_{x\to \infty} \frac{v_{\delta}(x)}{u_{\delta}(x)}=\lim_{x\to \infty} \frac{|x|^{1+f'(0)}}{u_{\delta}(x)}=0.
\end{equation*}
Hence, there exists $R_{0}>R$ such that
\begin{equation*}
v_{\delta}<u_{\delta}\quad\mbox{in}~~B_{R_0}^c.
\end{equation*}
By the comparison principle, we also have
\begin{equation*}
v_{\delta}<u_{\delta}\quad\mbox{in}~~B_{R_0}\backslash B_R.
\end{equation*}
That is,
\begin{equation}\label{e3.16}
v_{\delta}<u_{\delta}\quad\mbox{in}~~ B_R^c.
\end{equation}
By \Cref{le2.3},
\begin{equation*}
u_{\delta}(x)=\delta u(\delta^{-1}x)\leq C\delta^{\alpha}|x|^{1-\alpha},~\forall ~x\in B_R^c.
\end{equation*}
In addition, by the uniform estimates for $v_{\delta}$ (see (i) in \Cref{le8.1}), up to a subsequence, $v_{\delta}\to v$ in $C^{1}_{loc}(B_R^c)$ for some $v$. Let $\delta \to 0$ in \cref{e3.16}, we have
\begin{equation}\label{e6.6}
v\leq 0\quad\mbox{in}~~ B_R^c,
\end{equation}
which is a contradiction since $v'(R)=1>0$.

Next, we consider the case $c=0$. By \Cref{le8.1}, for any $\varepsilon>0$, there exists a positive radial solution $v^{\varepsilon}$ of \cref{e1.6}--\cref{e1.6-2}--\cref{e1.6-t} in $B_R^c$ with $v^{\varepsilon}(R)=0$, $(v^{\varepsilon})'(R)=C_{\varepsilon}$ for some $C_{\varepsilon}>0$ and
\begin{equation*}
\lim_{x\to \infty} \frac{v^{\varepsilon}(x)}{|x|^{1+f'(0)}}=\varepsilon.
\end{equation*}
Hence, by noting $c=0$,
\begin{equation*}
\liminf_{x\to \infty} \frac{u(x)}{v^{\varepsilon}(x)}=0.
\end{equation*}
Then there exist sequences of $R_i\rightarrow \infty$ increasingly and $x_i\in \partial B_{R_i}$ such that
\begin{equation*}
u(x_i)\leq \varepsilon v^{\varepsilon}(x_i), \quad ~\forall ~i\geq 1.
\end{equation*}
By applying the Harnack inequality to $u$ and choosing $\varepsilon$ sufficiently small, we obtain
\begin{equation*}
u(x)\leq Cu(x_i)\leq C\varepsilon v^{\varepsilon}(x_i)
\leq v^{\varepsilon}(x_i)=v^{\varepsilon}(x),~\forall ~x\in \partial B_{R_i},~\forall ~i\geq 2.
\end{equation*}
Thus, from the comparison principle,
\begin{equation*}
u\leq v^{\varepsilon}+\max_{\partial B_R} u~~\mbox{ in}~~B_R^c.
\end{equation*}
By letting $\varepsilon\to 0$ (equivalently $C_{\varepsilon} \to 0$), we infer that $v^{\varepsilon} \to 0$ and
\begin{equation*}
u\leq \max_{\partial B_R} u~~\mbox{ in}~~B_R^c.
\end{equation*}
That is, $u$ is bounded. Hence, from \Cref{th1.1}, there exists a constant $K$ such that \cref{e3.14} holds.

Finally, we consider the case $0<c<+\infty$.  By the assumption, we have for the radial solution $v^{\varepsilon}$, with $\varepsilon=c$, that
\begin{equation*}
\liminf_{x\to \infty} \frac{u(x)}{v^{c}(x)}=1.
\end{equation*}
Then, for any $0<\varepsilon< c$, there exist sequences of $R_i\rightarrow \infty$ increasingly and $x_i\in \partial B_{R_i}$ such that
\begin{equation*}
v^{c-\varepsilon}\leq u ~~\mbox{ in } B^c_{R_1}, \quad
u(x_i)\leq v^{c+\varepsilon}(x_i), \quad ~\forall ~i\geq 1.
\end{equation*}

Let $w=u-v^{c-\varepsilon}$ and then as before (see \cref{e2.1}), $w$ is a solution of the following linear uniformly elliptic equation:
\begin{equation*}\label{e5.10}
a^{ij}w_{ij}+a^iw_i=0\quad\mbox{in}~~ \bar{B}_R^c,
\end{equation*}
with
\begin{equation*}
|a^i(x)|\leq C|x|^{-1}.
\end{equation*}
Then by applying the Harnack inequality to $w$ on $\partial B_{R_i}$ ($i\geq 2$ large enough), we have
\begin{equation*}
\begin{aligned}
u(x)-v^{c-\varepsilon}(x)=&\ w(x)\leq Cw(x_i)
= C\left(u(x_i)-v^{c-\varepsilon}(x_i) \right)\\
\leq&\  C\left(v^{c+\varepsilon}(x_i)-v^{c-\varepsilon}(x_i)\right)
\leq v^{C\varepsilon}(x_i)\\
=&\ v^{C\varepsilon}(x),\quad\forall ~x\in \partial B_{R_i}.
\end{aligned}
\end{equation*}
That is,
\begin{equation*}
 v^{c-\varepsilon}\leq u\leq v^{c+C\varepsilon}~~\mbox{ on}~~\partial B_{R_i},~\forall ~i\geq 2 ~~\mbox{ large enough}.
\end{equation*}
From the comparison principle,
\begin{equation*}
v^{c-\varepsilon} \leq u\leq v^{c+C\varepsilon}+\max_{\partial B_{R}} u~~\mbox{ in}~~B_{R}^c.
\end{equation*}
By the definitions of $v^{c-\varepsilon}$ and $v^{c+C\varepsilon}$, there exists $R_0$ (depending on $\varepsilon$) such that
\begin{equation*}
  \left|\frac{v^{c-\varepsilon}(x)}{|x|^{1+f'(0)}}-c\right|\leq 2\varepsilon, \quad
\left|\frac{v^{c+C\varepsilon}(x)}{|x|^{1+f'(0)}}-c\right|\leq 2C\varepsilon, ~\forall ~x\in B_{R_0}^c.
\end{equation*}
Hence,
\begin{equation*}
\left|\frac{u(x)}{|x|^{1+f'(0)}}-c\right|\leq C\varepsilon, ~\forall ~x\in B_{R_0}^c.
\end{equation*}
That is,
\begin{equation*}
\lim_{x\to \infty} \frac{u(x)}{|x|^{1+f'(0)}}=c,
\end{equation*}
as we wanted to show. ~\qed\bigskip

Next, we consider the case $f'(0)<-1$. We have
\begin{lemma}\label{le3.3}
Let $u$ define a positive solution to \cref{e1.6}--\cref{e1.6-2}--\cref{e1.6-t} in $B_R^c$ with $f'(0)<-1$. Then there exists a constant $K$ such that
\begin{equation}\label{e3.17}
\lim_{x\rightarrow \infty} u(x)=K.
\end{equation}
\end{lemma}
\proof
The proof follows the same ideas as in \Cref{le3.2}. Thus, if $\lim_{x \to \infty} u(x) = +\infty$, it can be shown in an analogous way that \cref{e6.6} holds; that is,
\begin{equation*}
v \leq 0 \quad \text{in}~ B_R^c,
\end{equation*}
for some radial function $v$ with $v'(R) = 1$, which again leads to a contradiction. Therefore, by \Cref{th1.1}, we obtain \cref{e3.17}.~\qed\bigskip

\section{Appendix}\label{S7}
In \cite{MR4684292}, a series of uniqueness results were established for elliptic Weingarten surfaces in the symmetric case, that is, when $\kappa_2=f(\kappa_1)$ where $f\circ f$ is the identity. In this section, we show that the techniques developed in this article allow these results to be extended to the general case $\kappa_2=f(\kappa_1)$, with $f(c)=c\in\mathbb{R}$, where no symmetry assumption on $f$ is imposed, and even its regularity can be weakened.

We begin by extending Lemma 2.10 in \cite{MR4684292}. More precisely, we obtain the following result.	
\begin{lemma}\label{le9.1}
Let $\Sigma$ be an elliptic Weingarten surface satisfying \cref{e1.6} where $f\in C^{0,1}_{loc}[c,a)$ for some $a>c\geq 0$ and
\begin{equation*}
-\infty<\inf_{t\in [c,a']}f'(t)\leq \sup_{t\in [c,a']}f'(t)<0,~\forall ~a'<a.
\end{equation*}
Assume that $\Sigma$ is a multigraph, that is, the third coordinate of its unit normal satisfies $N_3\geq 0$.

Then either $N_3\equiv 0$, or $N_3>0$ on $\Sigma$. If $N_3\equiv 0$, then $\Sigma$ is a portion of a plane or a cylinder.
\end{lemma}
\proof Since the fully nonlinear operator $F$ (see \cref{e4.2}) is not smooth, the original proof cannot be applied; in particular, (2.10) in \cite{MR4684292} is no longer available. We provide an alternative proof here.

As in the proof of \cite[Lemma 2.10]{MR4684292}, we assume that there exists $q_0\in \Sigma$ satisfying $N_3(q_0)=0$. Without loss of generality, we can suppose that $q_0$ is the origin and its unit normal is $N(q_0) = (1,0,0)$. Let $C$ denote the vertical plane or the vertical cylinder (depending on whether $c=0$ or $c\neq 0$) that is tangent to $\Sigma$ at $q_0$, with the same orientation. Then $\Sigma$ and $C$ can be seen around the origin as graphs $x = h^i(y,z), i = 1,2$ over their common tangent plane, and $h^1, h^2$ are solutions to the same fully nonlinear elliptic PDE. Note that $h^2\equiv 0$ (the vertical plane case) or $h^2$ is independent of $z$ (the vertical cylinder case).

Let $v=\partial h^1/\partial z$. Then $v$ belongs to the Pucci's class $S(0)$ (see for instance the second paragraph of the proof of \cite[Lemma 4.1]{MR3246039}). Note that $v\leq 0$ (i.e., $N_3\geq 0$) and $v(0)=0$ (i.e., $N_3(q_0)=0$). By the strong maximum principle, $v\equiv 0$. That is, $h^1$ is independent of $z$ as well. Then, the function $h=h^1-h^2$ is a solution of the following linear elliptic equation
\begin{equation*}
a^{ij}h_{ij}+a^ih_i=0
\end{equation*}
for some $a^{ij}$ and $a^{i}$ (depending on $h_1,h_2$). As $h$ is independent of $z$, the equation above reduces to
\begin{equation*}
a^{11}h_{11}+a^1h_1=0.
\end{equation*}
Since $C$ is tangent to $\Sigma$ at the origin, $h(0)=0$ and $h_1(0)=0$. Hence, by the ODE theory, $h\equiv 0$, and $\Sigma$ is a piece of the plane or the cylinder. ~\qed\bigskip

Theorem 4.2 in \cite{MR4684292} proves a Bernstein theorem for elliptic Weingarten surfaces. Thanks to \Cref{le9.1} (instead of Lemma 2.10 in \cite{MR4684292}) this result can be generalized. Specifically, it can be shown that there are no complete elliptic Weingarten multigraphs with $f(c)=c\neq 0$ and bounded second fundamental form:
\begin{theorem}\label{le9.2}
There is no complete multigraph with bounded second fundamental form satisfying \cref{e1.6} for some $f\in C^{1}_{loc}[c,a)$, with $a>c> 0$ and $f'<0$.
\end{theorem}
\proof The proof follows verbatim that of \cite[Theorem 4.2]{MR4684292}. The only step in which the smoothness of the fully nonlinear elliptic operator $F$ (see \cref{e4.2}) is required is in the linearization of the operator (see \cite[(4.3), p. 1909]{MR4684292}). Observe that the linearization is performed at the cylinder, where
$F$ is indeed as smooth as
$f$, namely of class  $C^1$, since the lack of smoothness of
$F$ occurs only at umbilical points of surfaces. ~\qed\bigskip

We now extend Theorem 5.2 in \cite{MR4684292} by establishing the existence of curvature estimates for general functions $f$:
\begin{theorem}\label{th8.1}
Let $\Sigma$ be a complete surface in $\mathbb{R}^3$, possibly with boundary $\partial \Sigma$, and whose Gauss map image $N(\Sigma)$ is contained in an open hemisphere of $\mathbb{S}^2$. Assume that $\Sigma$ satisfies a uniformly elliptic Weingarten equation \cref{e1.6}--\cref{e1.6-2} where $f\in C^{0,1}_{loc}[c,+\infty)$ for some $c\geq 0$.

Then, for every $d > 0$ there exists a constant $C = C(\Lambda,c,d)$ such that for each $p\in \Sigma$ with $d(p,\partial \Sigma) \geq d$, it holds
\begin{equation*}
|\sigma(p)| \leq  C.
\end{equation*}
Here, $d$ and $|\sigma|$ denote, respectively, the distance function in $\Sigma$ and the norm of
the second fundamental form of $\Sigma$.

In particular, if $\Sigma$ is of minimal type, i.e., $f(0)=0$, we have
\begin{equation}\label{e9.4}
|\sigma(p)| \leq  \frac{C}{d(p,\partial \Sigma)},
\end{equation}
where $C$ depends only on $\Lambda$.
\end{theorem}
\proof Again, we can use essentially the same proof as in \cite[Theorem 5.2]{MR4684292}. So, we will focus only on the parts of the proof that change due to the lack of differentiability of the associated PDE. The condition \cref{e1.6}  in \cite{MR4684292} is used  to ensure that the fully nonlinear elliptic operator $F$ (see \cref{e4.2}) is $C^1$. The authors then apply Nirenberg's $C^{2,\alpha}$ estimate. In fact, this estimate remains valid even when the operator $F$ is not smooth (see \cite[Corollary 1.2]{MR5005582}).

We use the same notation as in the proof of \cite[Theorem 5.2]{MR4684292} and only point out the main differences below. As in the proof of \cite[Theorem 5.2]{MR4684292}, we prove the result by contradiction and take a sequence of blow-ups. In our setting, equation (5.4) in \cite{MR4684292} is replaced by the following:
\begin{equation}\label{e9.1}
\kappa_2=\tilde f_n(\kappa_1), \quad \tilde{f}_n(t):=\lambda_n^{-1}f_n(\lambda_n t),
\end{equation}
where $\lambda_n\to +\infty$. Note that $\tilde{f}_n$ are uniformly elliptic as well (with the same $\Lambda$), i.e.,
\begin{equation}\label{e9.3}
 -\frac{1}{\Lambda}\leq \tilde{f}_n'(t)\leq  -\Lambda,~\forall ~t\geq c_n,
\end{equation}
where $c_n=\lambda_n^{-1}c$ and $c$ is the common umbilical constant for $f_n$.

In addition, we have the uniform curvature estimate (see \cite[p. 1913]{MR4684292}):
\begin{equation*}
|\hat{\sigma}_n|\leq 2\quad\mbox{in}~~\hat{D}_n.
\end{equation*}
And by representing the surfaces locally as graphs of $v_n$, we have the following uniform estimate:
\begin{equation*}
\|v_n\|_{C^2(B_{\delta_0})}\leq \mu_0.
\end{equation*}
Hence, by \Cref{le2.6}, $v_n$ are solutions of fully nonlinear uniformly elliptic equations:
\begin{equation*}
F_n(Dv_n,D^2v_n)=0\quad\mbox{in}~~B_{\delta_0}.
\end{equation*}
Furthermore, by the uniform ellipticity of $\tilde{f}_n$ and noting $f_n(c)=c$,
\begin{equation*}
\begin{aligned}
|F_n(0,0)|=|\tilde{f}_n(0)|=\lambda_n^{-1}|f_n(0)|=\lambda_n^{-1}|f_n(c)-f'_n(\xi_n)c|
=\lambda_n^{-1}|c-f'_n(\xi_n)c|\leq C\lambda_n^{-1},
\end{aligned}
\end{equation*}
where $\xi_n\in [0,c]$ and $C$ depends only on $\Lambda$ and $c$.

Therefore, by the interior $C^{2,\alpha}$ regularity for $v_n$ (see \Cref{le2.2}), we have the same uniform estimate as in \cite[(5.13)]{MR4684292} for $v_n$:
\begin{equation}\label{e9.2}
\|v_n\|_{C^{2,\alpha}(B_{\delta'})}\leq C',~\forall ~n\geq 1.
\end{equation}
Then there exists $v^0\in C^{2,\alpha}(B_{\delta'})$ such that (up to a subsequence)
\begin{equation*}
v_n\to v^0\quad\mbox{in}~~C^{2}(B_{\delta'}).
\end{equation*}

Note that $\tilde{f}_n$ are Lipschitz continuous with a uniform Lipschitz constant (see \cref{e9.3}). In addition,
\begin{equation*}
  \tilde{f}_n(c_n)=\lambda_n^{-1}f_n(c)=\lambda_n^{-1}c=c_n
\end{equation*}
and $c_n\to 0$. Hence, up to a subsequence, there exists $f^0\in C^{0,1}_{loc}[0,+\infty)$ with $f^0(0)=0$ such that
\begin{equation*}
\tilde{f}_n\to f^0\quad\mbox{in}~~C_{loc}[0,+\infty).
\end{equation*}
Then $v^0$ is a solution of
\begin{equation*}
\kappa_2=f^0(\kappa_1)\quad\mbox{in}~~B_{\delta'}.
\end{equation*}
That is, the graph of $v^0$ is a uniformly elliptic Weingarten surface of minimal type. Thus, its Gauss map is a quasiregular mapping (see \Cref{S2}). It is then straightforward to verify that all the properties (P1)–(P5) in \cite[pp. 1915--1916]{MR4684292} hold. The remaining part of the proof is the same as that of \cite[Theorem 5.2]{MR4684292} and we arrive at a contradiction.

If $\Sigma$ is of minimal type, i.e., $f(0)=0$, we can obtain \cref{e9.4} by a standard scaling argument.~\qed\bigskip

As a consequence of the two preceding theorems, Theorem 6.2 in \cite{MR4684292} can be generalized, with the same proof, to a more general setting:
\begin{theorem}\label{th9.2}
Let $\Sigma$ be a complete multigraph in $\mathbb{R}^3$ that satisfies an elliptic Weingarten equation \cref{e1.6} with $f\in C^{0,1}_{loc}[0,a)$, $f(0)=0$ and
\begin{equation*}
-\infty<\inf_{t\in [c,a']}f'(t)\leq \sup_{t\in [c,a']}f'(t)<0,~\forall ~a'<a.
\end{equation*}
Assume~\\
(i) $a\in \mathbb{R}$ and
\begin{equation*}
f(t)\to -\infty\quad\mbox{as }~~t\to a.
\end{equation*}
or ~\\
(ii) $a=+\infty$ and for some $b\in \mathbb{R}$,
\begin{equation*}
f(t)\to b\quad\mbox{as }~~t\to +\infty.
\end{equation*}
Then $\Sigma$ is a plane.
\end{theorem}

With the aid of the results above, we conclude that, in addition to not requiring any symmetry assumption on  $f$, the main results, Theorems B and C in \cite[pp. 1889--1890]{MR4684292}, hold for functions $f\in C^1$, and Theorem D in \cite[p. 1890]{MR4684292} holds for $f\in C^{0,1}$.


\printbibliography

@book{MR1351007,
	TITLE        = {Fully nonlinear elliptic equations},
	AUTHOR       = {Caffarelli, Luis A. and Cabr\'{e}, Xavier},
	YEAR         = {1995},
	PUBLISHER    = {American Mathematical Society, Providence, RI},
	SERIES       = {American Mathematical Society Colloquium Publications},
	VOLUME       = {43},
	PAGES        = {vi+104},
	DOI          = {10.1090/coll/043},
	ISBN         = {0-8218-0437-5},
	URL          = {https://doi.org/10.1090/coll/043},
	MRCLASS      = {35J60 (35-01 35B45 35B65 35Dxx)},
	MRNUMBER     = {1351007},
	MRREVIEWER   = {P. Lindqvist}
}

@misc{lian2025liouville,
      title={Liouville theorems for fully nonlinear elliptic equations on half spaces},
      author={Yuanyuan Lian},
      year={2025},
      eprint={2511.16152},
      archivePrefix={arXiv},
      primaryClass={math.AP},
      url={https://arxiv.org/abs/2511.16152},
}

@article{MR454854,
	TITLE        = {Equations of mean curvature type in {$2$} independent variables},
	AUTHOR       = {Simon, Leon},
	YEAR         = {1977},
	JOURNAL      = {Pacific J. Math.},
	VOLUME       = {69},
	NUMBER       = {1},
	PAGES        = {245--268},
	ISSN         = {0030-8730,1945-5844},
	URL          = {http://projecteuclid.org/euclid.pjm/1102817107},
	FJOURNAL     = {Pacific Journal of Mathematics},
	MRCLASS      = {49F10 (35J20)},
	MRNUMBER     = {454854},
	MRREVIEWER   = {Klaus\ Steffen}
}

@article{MR975123,
	TITLE        = {A maximum principle at infinity for minimal surfaces and applications},
	AUTHOR       = {Langevin, R\'{e}mi and Rosenberg, Harold},
	YEAR         = {1988},
	JOURNAL      = {Duke Math. J.},
	VOLUME       = {57},
	NUMBER       = {3},
	PAGES        = {819--828},
	DOI          = {10.1215/S0012-7094-88-05736-5},
	ISSN         = {0012-7094},
	URL          = {https://doi.org/10.1215/S0012-7094-88-05736-5},
	FJOURNAL     = {Duke Mathematical Journal},
	MRCLASS      = {53A10 (35J60 49F10 58E12)},
	MRNUMBER     = {975123},
	MRREVIEWER   = {Marco Rigoli}
}

@book {MR1013786,
    AUTHOR = {Hopf, Heinz},
     TITLE = {Differential geometry in the large},
    SERIES = {Lecture Notes in Mathematics},
    VOLUME = {1000},
   EDITION = {Second},
      NOTE = {Notes taken by Peter Lax and John W. Gray,
              With a preface by S. S. Chern,
              With a preface by K. Voss},
 PUBLISHER = {Springer-Verlag, Berlin},
      YEAR = {1989},
     PAGES = {viii+184},
      ISBN = {3-540-51497-X},
   MRCLASS = {53-01 (01A75 53A05)},
  MRNUMBER = {1013786},
       DOI = {10.1007/3-540-39482-6},
       URL = {https://doi.org/10.1007/3-540-39482-6},
}

@article {MR4684292,
    AUTHOR = {Fern\'{a}ndez, Isabel and G\'{a}lvez, Jos\'{e} A. and Mira, Pablo},
     TITLE = {Quasiconformal {G}auss maps and the {B}ernstein problem for
              {W}eingarten multigraphs},
   JOURNAL = {Amer. J. Math.},
  FJOURNAL = {American Journal of Mathematics},
    VOLUME = {145},
      YEAR = {2023},
    NUMBER = {6},
     PAGES = {1887--1921},
      ISSN = {0002-9327},
   MRCLASS = {53A05 (30C65)},
  MRNUMBER = {4684292},
MRREVIEWER = {Jos\'{e} Miguel Manzano},
       DOI = {10.1353/ajm.2023.a913297},
       URL = {https://doi.org/10.1353/ajm.2023.a913297},
}

@article{MR4792754,
	TITLE        = {Quasiconformal mappings and a {B}ernstein type theorem over exterior domains in {$\Bbb R^2$}},
	AUTHOR       = {Li, Dongsheng and Liu, Rulin},
	YEAR         = {2024},
	JOURNAL      = {Calc. Var. Partial Differential Equations},
	VOLUME       = {63},
	NUMBER       = {8},
	PAGES        = {Paper No. 209, 13},
	DOI          = {10.1007/s00526-024-02808-3},
	ISSN         = {0944-2669},
	URL          = {https://doi.org/10.1007/s00526-024-02808-3},
	FJOURNAL     = {Calculus of Variations and Partial Differential Equations},
	MRCLASS      = {35B53 (35J60 35J96)},
	MRNUMBER     = {4792754}
}

@BOOK{MR1814364,
	title        = {Elliptic partial differential equations of second order},
	author       = {Gilbarg, David and Trudinger, Neil S.},
	year         = {2001},
	publisher    = {Springer-Verlag, Berlin},
	series       = {Classics in Mathematics},
	pages        = {xiv+517},
	isbn         = {3-540-41160-7},
	note         = {Reprint of the 1998 edition},
	mrclass      = {35-02 (35Jxx)},
	mrnumber     = {1814364}
}

@misc{LZ_asym,
	title        = {Asymptotic behavior for fully nonlinear elliptic equations in exterior domains},
	author       = {Lian, Yuanyuan and Zhang, Kai},
	year         = {2024},
	eprint       = {2401.05829},
	archivePrefix = {arXiv}
}

@article{MR4201294,
	TITLE        = {Maximal hypersurfaces over exterior domains},
	AUTHOR       = {Hong, Guanghao and Yuan, Yu},
	YEAR         = {2021},
	JOURNAL      = {Comm. Pure Appl. Math.},
	VOLUME       = {74},
	NUMBER       = {3},
	PAGES        = {589--614},
	DOI          = {10.1002/cpa.21929},
	ISSN         = {0010-3640},
	URL          = {https://doi.org/10.1002/cpa.21929},
	FJOURNAL     = {Communications on Pure and Applied Mathematics},
	MRCLASS      = {53A10 (49Q05 53A35)},
	MRNUMBER     = {4201294},
	MRREVIEWER   = {Andrew Bucki}
}

@article{MR452746,
	TITLE        = {A {H}\"older estimate for quasiconformal maps between surfaces in {E}uclidean space},
	AUTHOR       = {Simon, Leon},
	YEAR         = {1977},
	JOURNAL      = {Acta Math.},
	VOLUME       = {139},
	NUMBER       = {1-2},
	PAGES        = {19--51},
	DOI          = {10.1007/BF02392233},
	ISSN         = {0001-5962,1871-2509},
	URL          = {https://doi.org/10.1007/BF02392233},
	FJOURNAL     = {Acta Mathematica},
	MRCLASS      = {30C60 (49F10 53A10)},
	MRNUMBER     = {452746}
}

@article{MR4038557,
	TITLE        = {A {B}ernstein problem for special {L}agrangian equations in exterior domains},
	AUTHOR       = {Li, Dongsheng and Li, Zhisu and Yuan, Yu},
	YEAR         = {2020},
	JOURNAL      = {Adv. Math.},
	VOLUME       = {361},
	PAGES        = {106927, 29},
	DOI          = {10.1016/j.aim.2019.106927},
	ISSN         = {0001-8708},
	URL          = {https://doi.org/10.1016/j.aim.2019.106927},
	FJOURNAL     = {Advances in Mathematics},
	MRCLASS      = {35J60 (35J15 35J96 58J05)},
	MRNUMBER     = {4038557},
	MRREVIEWER   = {Chad R. Westphal}
}

@book {MR1929104,
    AUTHOR = {Hartman, Philip},
     TITLE = {Ordinary differential equations},
    SERIES = {Classics in Applied Mathematics},
    VOLUME = {38},
 PUBLISHER = {Society for Industrial and Applied Mathematics (SIAM),
              Philadelphia, PA},
      YEAR = {2002},
     PAGES = {xx+612},
      ISBN = {0-89871-510-5},
   MRCLASS = {34-01 (37-01)},
  MRNUMBER = {1929104},
       DOI = {10.1137/1.9780898719222},
       URL = {https://doi.org/10.1137/1.9780898719222},
}

@article {MR2652214,
    AUTHOR = {Aledo, Juan A. and Espinar, Jos\'e{} M. and G\'alvez, Jos\'e{}
              A.},
     TITLE = {The {C}odazzi equation for surfaces},
   JOURNAL = {Adv. Math.},
  FJOURNAL = {Advances in Mathematics},
    VOLUME = {224},
      YEAR = {2010},
    NUMBER = {6},
     PAGES = {2511--2530},
      ISSN = {0001-8708,1090-2082},
   MRCLASS = {53C42 (53A05)},
  MRNUMBER = {2652214},
MRREVIEWER = {Federico\ S\'anchez-Bringas},
       DOI = {10.1016/j.aim.2010.02.007},
       URL = {https://doi.org/10.1016/j.aim.2010.02.007},
}

@article {MR86338,
    AUTHOR = {Aleksandrov, A. D.},
     TITLE = {Uniqueness theorems for surfaces in the large. {I}},
   JOURNAL = {Vestnik Leningrad. Univ.},
  FJOURNAL = {Vestnik Leningrad. Univ.},
    VOLUME = {11},
      YEAR = {1956},
    NUMBER = {19},
     PAGES = {5--17},
   MRCLASS = {53.0X},
  MRNUMBER = {86338},
MRREVIEWER = {H.\ Busemann},
}

@article {MR74857,
    AUTHOR = {Chern, Shiing-Shen},
     TITLE = {On special {$W$}-surfaces},
   JOURNAL = {Proc. Amer. Math. Soc.},
  FJOURNAL = {Proceedings of the American Mathematical Society},
    VOLUME = {6},
      YEAR = {1955},
     PAGES = {783--786},
      ISSN = {0002-9939,1088-6826},
   MRCLASS = {53.0X},
  MRNUMBER = {74857},
MRREVIEWER = {L.\ W.\ Green},
       DOI = {10.2307/2032934},
       URL = {https://doi.org/10.2307/2032934},
}

@misc{EspinarMesa,
	title        = {In preparation},
	author       = {Espinar, Jos\'e{} M. and Mesa, Heber},
	year         = {2025},
}

@article {MR4182894,
    AUTHOR = {G\'alvez, Jos\'e{} A. and Mira, Pablo},
     TITLE = {Uniqueness of immersed spheres in three-manifolds},
   JOURNAL = {J. Differential Geom.},
  FJOURNAL = {Journal of Differential Geometry},
    VOLUME = {116},
      YEAR = {2020},
    NUMBER = {3},
     PAGES = {459--480},
      ISSN = {0022-040X,1945-743X},
   MRCLASS = {53A10 (53C42)},
  MRNUMBER = {4182894},
MRREVIEWER = {Martin\ L. P. Kilian},
       DOI = {10.4310/jdg/1606964415},
       URL = {https://doi.org/10.4310/jdg/1606964415},
}

@article {MR4237967,
    AUTHOR = {G\'alvez, Jos\'e{} A. and Mira, Pablo},
     TITLE = {Rotational symmetry of {W}eingarten spheres in homogeneous
              three-manifolds},
   JOURNAL = {J. Reine Angew. Math.},
  FJOURNAL = {Journal f\"ur die Reine und Angewandte Mathematik. [Crelle's
              Journal]},
    VOLUME = {773},
      YEAR = {2021},
     PAGES = {21--66},
      ISSN = {0075-4102,1435-5345},
   MRCLASS = {53C42 (53C30)},
  MRNUMBER = {4237967},
MRREVIEWER = {Marian-Ioan\ Munteanu},
       DOI = {10.1515/crelle-2020-0031},
       URL = {https://doi.org/10.1515/crelle-2020-0031},
}

@article {MR63082,
    AUTHOR = {Hartman, Philip and Wintner, Aurel},
     TITLE = {Umbilical points and {$W$}-surfaces},
   JOURNAL = {Amer. J. Math.},
  FJOURNAL = {American Journal of Mathematics},
    VOLUME = {76},
      YEAR = {1954},
     PAGES = {502--508},
      ISSN = {0002-9327,1080-6377},
   MRCLASS = {53.0X},
  MRNUMBER = {63082},
MRREVIEWER = {S.\ Chern},
       DOI = {10.2307/2372698},
       URL = {https://doi.org/10.2307/2372698},
}

@article {MR40042,
    AUTHOR = {Hopf, Heinz},
     TITLE = {\"Uber {F}l\"achen mit einer {R}elation zwischen den
              {H}auptkr\"ummungen},
   JOURNAL = {Math. Nachr.},
  FJOURNAL = {Mathematische Nachrichten},
    VOLUME = {4},
      YEAR = {1951},
     PAGES = {232--249},
      ISSN = {0025-584X,1522-2616},
   MRCLASS = {53.0X},
  MRNUMBER = {40042},
MRREVIEWER = {C.\ B.\ Allendoerfer},
       DOI = {10.1002/mana.3210040122},
       URL = {https://doi.org/10.1002/mana.3210040122},
}

@article {MR1262209,
    AUTHOR = {Rosenberg, Harold and Sa Earp, Ricardo},
     TITLE = {The geometry of properly embedded special surfaces in {${\bf
              R}^3$}, e.g., surfaces satisfying {$aH+bK=1$}, where {$a$} and
              {$b$} are positive},
   JOURNAL = {Duke Math. J.},
  FJOURNAL = {Duke Mathematical Journal},
    VOLUME = {73},
      YEAR = {1994},
    NUMBER = {2},
     PAGES = {291--306},
      ISSN = {0012-7094,1547-7398},
   MRCLASS = {53A10 (49Q05)},
  MRNUMBER = {1262209},
MRREVIEWER = {Jo\~ao\ Lucas Marques Barbosa},
       DOI = {10.1215/S0012-7094-94-07314-6},
       URL = {https://doi.org/10.1215/S0012-7094-94-07314-6},
}

@article {MR1738404,
    AUTHOR = {Sa Earp, Ricardo and Toubiana, Eric},
     TITLE = {Classification des surfaces de type {D}elaunay},
   JOURNAL = {Amer. J. Math.},
  FJOURNAL = {American Journal of Mathematics},
    VOLUME = {121},
      YEAR = {1999},
    NUMBER = {3},
     PAGES = {671--700},
      ISSN = {0002-9327,1080-6377},
   MRCLASS = {53A05 (53A10)},
  MRNUMBER = {1738404},
MRREVIEWER = {Rabah\ Souam},
       URL =
              {http://muse.jhu.edu/journals/american_journal_of_mathematics/v121/121.3sa_earp.pdf},
}

@article {MR1364263,
    AUTHOR = {Sa Earp, Ricardo and Toubiana, Eric},
     TITLE = {Sur les surfaces de {W}eingarten sp\'eciales de type minimal},
   JOURNAL = {Bol. Soc. Brasil. Mat. (N.S.)},
  FJOURNAL = {Boletim da Sociedade Brasileira de Matem\'atica. Nova S\'erie},
    VOLUME = {26},
      YEAR = {1995},
    NUMBER = {2},
     PAGES = {129--148},
      ISSN = {0100-3569},
   MRCLASS = {53A10},
  MRNUMBER = {1364263},
MRREVIEWER = {Karsten\ Grosse-Brauckmann},
       DOI = {10.1007/BF01236989},
       URL = {https://doi.org/10.1007/BF01236989},
}

@article {MR108824,
    AUTHOR = {Voss, K.},
     TITLE = {\"Uber geschlossene {W}eingartensche {F}l\"achen},
   JOURNAL = {Math. Ann.},
  FJOURNAL = {Mathematische Annalen},
    VOLUME = {138},
      YEAR = {1959},
     PAGES = {42--54},
      ISSN = {0025-5831,1432-1807},
   MRCLASS = {53.00},
  MRNUMBER = {108824},
       DOI = {10.1007/BF01369665},
       URL = {https://doi.org/10.1007/BF01369665},
}

@article {MR683761,
    AUTHOR = {Jorge, Luqu\'esio P. and Meeks, III, William H.},
     TITLE = {The topology of complete minimal surfaces of finite total
              {G}aussian curvature},
   JOURNAL = {Topology},
  FJOURNAL = {Topology. An International Journal of Mathematics},
    VOLUME = {22},
      YEAR = {1983},
    NUMBER = {2},
     PAGES = {203--221},
      ISSN = {0040-9383},
   MRCLASS = {53A10 (53C42)},
  MRNUMBER = {683761},
MRREVIEWER = {Chi\ Cheng\ Chen},
       DOI = {10.1016/0040-9383(83)90032-0},
       URL = {https://doi.org/10.1016/0040-9383(83)90032-0},
}

@article {MR730928,
    AUTHOR = {Schoen, Richard M.},
     TITLE = {Uniqueness, symmetry, and embeddedness of minimal surfaces},
   JOURNAL = {J. Differential Geom.},
  FJOURNAL = {Journal of Differential Geometry},
    VOLUME = {18},
      YEAR = {1983},
    NUMBER = {4},
     PAGES = {791--809},
      ISSN = {0022-040X,1945-743X},
   MRCLASS = {53A10 (58E12)},
  MRNUMBER = {730928},
MRREVIEWER = {V.\ M.\ Miklyukov},
       URL = {http://projecteuclid.org/euclid.jdg/1214438183},
}

@article {MR4417394,
    AUTHOR = {G\'alvez, Jos\'e{} A. and Mira, Pablo and Tassi, Marcos P.},
     TITLE = {A quasiconformal {H}opf soap bubble theorem},
   JOURNAL = {Calc. Var. Partial Differential Equations},
  FJOURNAL = {Calculus of Variations and Partial Differential Equations},
    VOLUME = {61},
      YEAR = {2022},
    NUMBER = {4},
     PAGES = {Paper No. 129, 20},
      ISSN = {0944-2669,1432-0835},
   MRCLASS = {53A10 (35J60 53C42)},
  MRNUMBER = {4417394},
MRREVIEWER = {Jo\~ao\ Lucas Marques Barbosa},
       DOI = {10.1007/s00526-022-02222-7},
       URL = {https://doi.org/10.1007/s00526-022-02222-7},
}

@article {MR4942311,
    AUTHOR = {Nelli, B. and Pipoli, G. and Tassi, M. P.},
     TITLE = {Special {W}eingarten surfaces with planar convex boundary},
   JOURNAL = {Commun. Contemp. Math.},
  FJOURNAL = {Communications in Contemporary Mathematics},
    VOLUME = {27},
      YEAR = {2025},
    NUMBER = {9},
     PAGES = {Paper No. 2550014, 25},
      ISSN = {0219-1997,1793-6683},
   MRCLASS = {53A10 (53C42)},
  MRNUMBER = {4942311},
       DOI = {10.1142/S0219199725500142},
       URL = {https://doi.org/10.1142/S0219199725500142},
}

@article {MR4557820,
    AUTHOR = {Fern\'andez, Isabel and Mira, Pablo},
     TITLE = {Elliptic {W}eingarten surfaces: singularities, rotational
              examples and the halfspace theorem},
   JOURNAL = {Nonlinear Anal.},
  FJOURNAL = {Nonlinear Analysis. Theory, Methods \& Applications. An
              International Multidisciplinary Journal},
    VOLUME = {232},
      YEAR = {2023},
     PAGES = {Paper No. 113244, 27},
      ISSN = {0362-546X,1873-5215},
   MRCLASS = {53A10 (35J15 35J60 53C42)},
  MRNUMBER = {4557820},
       DOI = {10.1016/j.na.2023.113244},
       URL = {https://doi.org/10.1016/j.na.2023.113244},
}

@article {MR4089078,
    AUTHOR = {L\'opez, Rafael and P\'ampano, \'Alvaro},
     TITLE = {Classification of rotational surfaces in {E}uclidean space
              satisfying a linear relation between their principal
              curvatures},
   JOURNAL = {Math. Nachr.},
  FJOURNAL = {Mathematische Nachrichten},
    VOLUME = {293},
      YEAR = {2020},
    NUMBER = {4},
     PAGES = {735--753},
      ISSN = {0025-584X,1522-2616},
   MRCLASS = {53A05},
  MRNUMBER = {4089078},
MRREVIEWER = {Paulo\ Alexandre\ Sousa},
       DOI = {10.1002/mana.201800235},
       URL = {https://doi.org/10.1002/mana.201800235},
}

@article {MR906393,
    AUTHOR = {White, Brian},
     TITLE = {Complete surfaces of finite total curvature},
   JOURNAL = {J. Differential Geom.},
  FJOURNAL = {Journal of Differential Geometry},
    VOLUME = {26},
      YEAR = {1987},
    NUMBER = {2},
     PAGES = {315--326},
      ISSN = {0022-040X,1945-743X},
   MRCLASS = {53A10 (53A07)},
  MRNUMBER = {906393},
MRREVIEWER = {R.\ Osserman},
       URL = {http://projecteuclid.org/euclid.jdg/1214441372},
}

@article {MR94452,
    AUTHOR = {Huber, Alfred},
     TITLE = {On subharmonic functions and differential geometry in the
              large},
   JOURNAL = {Comment. Math. Helv.},
  FJOURNAL = {Commentarii Mathematici Helvetici},
    VOLUME = {32},
      YEAR = {1957},
     PAGES = {13--72},
      ISSN = {0010-2571,1420-8946},
   MRCLASS = {30.00 (31.00)},
  MRNUMBER = {94452},
MRREVIEWER = {E.\ F.\ Beckenbach},
       DOI = {10.1007/BF02564570},
       URL = {https://doi.org/10.1007/BF02564570},
}

@article {MR3069540,
    AUTHOR = {v. Neumann, J.},
     TITLE = {\"Uber einen hilfssatz der variationsrechnung.},
   JOURNAL = {Abh. Math. Sem. Univ. Hamburg},
  FJOURNAL = {Abhandlungen aus dem Mathematischen Seminar der Universit\"at
              Hamburg},
    VOLUME = {8},
      YEAR = {1931},
    NUMBER = {1},
     PAGES = {28--31},
      ISSN = {0025-5858,1865-8784},
   MRCLASS = {99-04},
  MRNUMBER = {3069540},
       DOI = {10.1007/BF02940985},
       URL = {https://doi.org/10.1007/BF02940985},
}

@article{zbMATH02587131,
 author = {Rad{\'o}, T.},
 title = {Geometrische {Betrachtungen} {\"u}ber zweidimensionale regul{\"a}re {Variationsprobleme}.},
 fjournal = {Acta Litterarum ac Scientiarum. Regiae Universitatis Hungaricae Francisco-Josephinae. Sectio Scientiarum Mathematicarum},
 journal = {Acta Litt. Sci. Szeged},
 volume = {2},
 pages = {228--253},
 year = {1926},
 zbMATH = {2587131},
 JFM = {52.0507.04}
}

@article {MR2669367,
    AUTHOR = {Rosenberg, Harold and Souam, Rabah and Toubiana, Eric},
     TITLE = {General curvature estimates for stable {$H$}-surfaces in
              3-manifolds and applications},
   JOURNAL = {J. Differential Geom.},
  FJOURNAL = {Journal of Differential Geometry},
    VOLUME = {84},
      YEAR = {2010},
    NUMBER = {3},
     PAGES = {623--648},
      ISSN = {0022-040X,1945-743X},
   MRCLASS = {53A10},
  MRNUMBER = {2669367},
MRREVIEWER = {Rafael\ L\'opez},
       URL = {http://projecteuclid.org/euclid.jdg/1279114303},
}

@article {MR81416,
    AUTHOR = {Gilbarg, D. and Serrin, James},
     TITLE = {On isolated singularities of solutions of second order
              elliptic differential equations},
   JOURNAL = {J. Analyse Math.},
  FJOURNAL = {Journal d'Analyse Math\'{e}matique},
    VOLUME = {4},
      YEAR = {1955/56},
     PAGES = {309--340},
      ISSN = {0021-7670,1565-8538},
   MRCLASS = {35.0X},
  MRNUMBER = {81416},
MRREVIEWER = {R.\ Finn},
       DOI = {10.1007/BF02787726},
       URL = {https://doi.org/10.1007/BF02787726},
}

@article{MR3246039,
	TITLE        = {Boundary regularity for viscosity solutions of fully nonlinear elliptic equations},
	AUTHOR       = {Silvestre, Luis and Sirakov, Boyan},
	YEAR         = {2014},
	JOURNAL      = {Comm. Partial Differential Equations},
	VOLUME       = {39},
	NUMBER       = {9},
	PAGES        = {1694--1717},
	DOI          = {10.1080/03605302.2013.842249},
	ISSN         = {0360-5302},
	URL          = {https://doi.org/10.1080/03605302.2013.842249},
	FJOURNAL     = {Communications in Partial Differential Equations},
	MRCLASS      = {35J60 (35B65 35D40 35J25 35J67)},
	MRNUMBER     = {3246039},
	MRREVIEWER   = {Leonard Monsaingeon}
}

@misc{lian2020pointwise,
	title        = {Pointwise Regularity for Fully Nonlinear Elliptic Equations in General Forms},
	author       = {Yuanyuan Lian and Lihe Wang and Kai Zhang},
	year         = {2020},
	eprint       = {2012.00324v3},
	archivePrefix = {arXiv},
	primaryClass = {math.AP}
}

@article{MR3299174,
	TITLE        = {Monge-{A}mp\`ere equation on exterior domains},
	AUTHOR       = {Bao, Jiguang and Li, Haigang and Zhang, Lei},
	YEAR         = {2015},
	JOURNAL      = {Calc. Var. Partial Differential Equations},
	VOLUME       = {52},
	NUMBER       = {1-2},
	PAGES        = {39--63},
	DOI          = {10.1007/s00526-013-0704-7},
	ISSN         = {0944-2669},
	URL          = {https://doi.org/10.1007/s00526-013-0704-7},
	FJOURNAL     = {Calculus of Variations and Partial Differential Equations},
	MRCLASS      = {35J96 (35B40 35D40 35J60 35J67)},
	MRNUMBER     = {3299174},
	MRREVIEWER   = {David A. Hartenstine}
}

@article {MR1085145,
    AUTHOR = {L\'opez, Francisco J. and Ros, Antonio},
     TITLE = {On embedded complete minimal surfaces of genus zero},
   JOURNAL = {J. Differential Geom.},
  FJOURNAL = {Journal of Differential Geometry},
    VOLUME = {33},
      YEAR = {1991},
    NUMBER = {1},
     PAGES = {293--300},
      ISSN = {0022-040X,1945-743X},
   MRCLASS = {53A10},
  MRNUMBER = {1085145},
MRREVIEWER = {Cun-Jin\ Sun},
       URL = {http://projecteuclid.org/euclid.jdg/1214446040},
}

@article {MR5005582,
    AUTHOR = {Zhang, Kai},
     TITLE = {Interior {$C^{2,\alpha}$}regularity for fully nonlinear
              uniformly elliptic equations in dimension two},
   JOURNAL = {J. Math. Anal. Appl.},
  FJOURNAL = {Journal of Mathematical Analysis and Applications},
    VOLUME = {558},
      YEAR = {2026},
    NUMBER = {1},
     PAGES = {Paper No. 130356},
      ISSN = {0022-247X,1096-0813},
   MRCLASS = {35D40},
  MRNUMBER = {5005582},
       DOI = {10.1016/j.jmaa.2025.130356},
       URL = {https://doi.org/10.1016/j.jmaa.2025.130356},
}
\end{document}